\documentclass[11pt]{article}
\pdfoutput=1

\usepackage{stokes-velocity-fespaces}
\usepackage{enumitem}
\usepackage{booktabs} 
\usepackage{subfig}   
\usepackage{varioref} 
\usepackage{multirow}

\usepackage{color} 

\usepackage[normalem]{ulem}

\usepackage{tikz}
\usetikzlibrary{calc} 

\usepackage{amsthm}

\makeatletter
\def\th@plain{%
  \thm@notefont{}
  \itshape 
}
\def\th@definition{%
  \thm@notefont{}
  \normalfont 
}
\makeatother

\newtheorem{theorem}{Theorem}
\newtheorem{lemma}[theorem]{Lemma}

\newtheorem{example}{Example}
\newtheorem{assumption}{Assumption}
\newtheorem{definition}{Definition}
\theoremstyle{definition}
\newtheorem{test}{Test}
\newtheorem{algorithm}{Algorithm}

\newtheorem{RemarkWithQed}{Remark}
\newenvironment{remark}
  {\RemarkWithQed}
  {\qed\endRemarkWithQed}

\usepackage[margin=2.5cm]{geometry}
\usepackage{hyperref}
\definecolor{darkblue}{rgb}{0.0,0.1,0.3}
\definecolor{darkgreen}{rgb}{0.0,0.35,0.15}
\definecolor{darkred}{rgb}{0.3,0.1,0.0}
\def\refColor{darkgreen}
\hypersetup{%
  colorlinks,
  linkcolor=\refColor,
  urlcolor=\refColor,
  anchorcolor=\refColor,
  citecolor=\refColor
}

\makeatletter
\newcommand{\myparagraph}{\@startsection
  {paragraph}
  {4}
  {0mm}
  {-0.5\baselineskip}
  {0.25\baselineskip}
  {\itshape}}
\newcounter{step}[theorem]
\newenvironment{step}[1]{%
  \stepcounter{step}\par\myparagraph{Step~\thestep.~#1}}{}
\makeatother

\makeatletter
\newcommand{\LeftEqNo}{\let\veqno\@@leqno}
\makeatother

\newcommand{\CiteSecondPaper}{\cite{Guillen-RGalvan:13b}\xspace}
\newcommand{\imgdir}{img}

\begin{document}
\title{On the stability of approximations for the Stokes problem using
  different finite element spaces for each component of the velocity%
  %
  \thanks{This work has been partially supported by project
    MTM2012-32325 (MINECO, Spain). Also the second author has been partially supported by the research group FQM-315 (Junta de Andaluc{\'i}a).}}  

\author{F. Guill\'en Gonz\'alez%
  \thanks{Departamento de Ecuaciones Diferenciales y An\'alisis
    Num\'erico and IMUS. Universidad de Sevilla. Aptdo. 1160, 41080 Sevilla
    (Spain).  {\tt guillen@us.es} }~
and J.R. Rodr\'iguez Galv\'an%
  \thanks{Departamento de Matem\'aticas. Universidad de
    C\'adiz. C.A.S.E.M. Pol{\'i}gono R{\'i}o San Pedro S/N, $11510$
    Puerto Real, C\'adiz (Spain).  {\tt rafael.rodriguez@uca.es} } }

\maketitle
 
\begin{abstract}
  This paper studies the stability of velocity-pressure mixed approximations of the
  Stokes problem when different finite element (FE) spaces
  for each component of the velocity field are considered. We consider
   some new combinations of continuous FE  reducing
    the number of degrees of freedom in some velocity
    components.  Although the resulting FE combinations are not stable in general,
     by using the Stenberg's macro-element technique, we show their stability in a wide family of meshes (namely, in uniformly unstructured meshes). Moreover,  a  post-processing is given in order to convert any mesh family  in  an uniformly unstructured mesh family. Finally, some $2D$ and
    $3D$ numerical simulations  are provided agree with the previous analysis.

  \par\smallskip
  \noindent \textbf{AMS classification (2010).} 35Q30, 65N12, 65N30, 76D07.
\end{abstract}

\tableofcontents

\section{Introduction}
\label{Introduction}

The main purpose of this work is to analyze the stability of
velocity-pressure mixed FE approximations of the
Stokes problem when the components of velocity are
approximated in different spaces.

As described below, this is an issue of great interest for the
approximation of some variants of the Navier-Stokes equations (in
particular for the approximation of the hydrostatic Navier-Stokes
equations governing the large scale ocean~\CiteSecondPaper) and, as far as we know, it
  has not been sufficiently studied in the literature. In fact,
  although stability of velocity-pressure mixed  FE approximations
  of the Stokes problem is a matter widely investigated in the last decades, almost all of these works fix  the same FE space for all
  components of the velocity field, being the pressure the only
  unknown that could be approximated by a different space.

Here we consider the classical Stokes problem in a 
domain $\domain\subset\Rset^d$ ($d=2$ or $3$),
\begin{align}
  - \Delta \ww + \grad\pp &= \ff \quad \In\domain,
  \label{eq:Stokes.a}
  \\
  \div \ww &= 0 \quad \In\domain,
  \label{eq:Stokes.b}
\end{align}
distinguishing the components of the velocity field as follows:
$\ww=(\uu,\vv)$, where $\uu:\domain \to \Rset^{d-1}$ and
$\vv:\domain\to\Rset$ (horizontal and vertical components,
respectively). Also $\pp:\domain\to\Rset$ is the pressure,
$\ff:\domain\to\Rset^d$ is the vector of external forces and, as
usual, $\Delta$ and $\div$ represent the Laplace and divergence
operators. For sake of simplicity, we take constant viscosity (equal
to one) and consider homogeneous Dirichlet boundary condition,
\begin{equation}
  \label{eq:Stokes.c}
  \ww = \textbf{0} \quad \On \partial\domain.
\end{equation}

Let us introduce usual notations: $L^2(\domain)$ represents the space of the
square integrable functions in $\domain$ and $(f,g)=\int_\domain fg$ denotes its scalar product, $H^m(\domain)$ is the space
of the functions whose distributional derivatives of order up to
$m\in\Nset$ belong to $L^2(\domain)$ and $H_0^1(\domain)$ is the
subspace of $H^1(\domain)$ whose functions vanish on
$\partial\domain$.

If one consider the bilinear forms $a(\ww,\bw)=(\grad\ww,\grad\bw)$
and $b(\bp,\ww)=(\bp,\div\ww)$, the variational formulation of the Stokes
problem~(\ref{eq:Stokes.a})--(\ref{eq:Stokes.c}) reads: find
$(\ww,\pp)\in \WW\times\PP$ such that:
\begin{align*}
  a(\ww,\bw) - b(\pp,\bw) &= \int_\Omega\ff\cdot\bw , \quad  \forall\,\bw\in\WW,
  \\
  b(\bp,\ww) &=0 ,\quad  \forall\,\bp\in\PP,
\end{align*}
where $\WW=H_0^1(\domain)^d$ and $\PP=L^2_0(\domain)=\big\{\pp \in L^2(\domain) /
\int_\domain \pp = 0 \big\}$. We assume
$\ff$ in $L^2(\domain)^d$.

It is well known (see for instance~\cite{boffi-brezzi-fortin_mixed_2013}) that
well-posedness of previous problem yields from the coercivity of
the bilinear form $a(\cdot,\cdot)$ and the inf-sup condition of
$b(\cdot,\cdot)$, which holds for the continuous problem on
$\WW\times\PP$. But if we consider the discrete problem, find
$(\wh,\ph)\in\Wh\times\Ph$ such that
\begin{align}
  \label{eq:discreteStokes.a}
  a(\wh,\bwh) - b(\ph,\bwh) &= \int_\Omega \ff\cdot\bwh, \quad \forall\,\bwh\in\Wh,
  \\
  \label{eq:discreteStokes.b}
  b(\bph,\wh) &=0 ,\quad \forall\,\bph\in\Ph,
\end{align}
where $\Wh\subset\WW$ and $\Ph\subset\WW$ are families of FE spaces,
although coercivity of $a(\cdot,\cdot)$ in $\Wh\times\Ph $ is automatically inherited,
the continuous inf-sup condition is not enough for well posedness
of~(\ref{eq:discreteStokes.a})--(\ref{eq:discreteStokes.b}).  In fact,
each choice of $\Wh$ and $\Ph$ requires checking the following
discrete counterpart of the inf-sup condition (also called LBB
condition): there exists $\beta>0$ independent on $h$ such that
\begin{equation}
  \label{eq:stokes-inf-sup}
  \beta\, \|\bph\|_{L^2} \le \sup_{0\neq\bwh\in\Wh}
  \frac{\int_\Omega (\div\bwh)\ph}{\|\grad\bwh\|_{L^2}} , \quad
   \forall\bph\in\Ph .
\end{equation}
In what follows, we will take $\Wh=\Uh\times\Vh$, where $\Uh$ and
$\Vh$ are, respectively, FE spaces approximating the (two
components of) horizontal velocity and the vertical velocity. If
$\domain\subset\Rset^2$, $\Uh$ is scalar and is denoted by $U_h$.
If no ambiguity, we denote by $\P1$ or $\P2$ the corresponding
continuous piecewise polynomial FE approximation, while $\P0$ denotes
discontinuous FE.  

The case where $\Uh$ and $\Vh$ are approximated by
the same FE space has been widely studied and several combinations
(which we denote by $\FEthreeSpaces{U_h}{\Vh}{\Ph}$, in the 2D
case) are known satisfying the discrete inf-sup
condition~(\ref{eq:stokes-inf-sup}). That is the case, for example, of
\FEthreeSpacesP{1,b}{1,b}{1} (usually known as \P1$\!\!+$bubble or
mini-element), \FEthreeSpacesP{2}{2}{1} and \FEthreeSpacesQ{2}{2}{1}
(Taylor-Hood elements), while some others do not
satisfy~(\ref{eq:stokes-inf-sup}) in general (see for instance
Section~\ref{sec:some-appr-with-P0-pressure} and references
therein). As usual, $\P{1,b}$ denotes $\P1$ continuous functions
enriched by a ``bubble'' function for each element and $\Q{k}$ are
continuous tensor $\P{k}\times\P{k}$ elements on quadrangular meshes.

This work is centered in the case where $\Uh$ and $\Vh$ are
approximated by different FE spaces, for which (as far as we know) only
the case \FEthreeSpacesP{2}{1}{0} has been addressed,
see~\cite{Stenberg:90}. Some new combinations
$\FEthreeSpaces\Uh\Vh\Ph$ are going to be studied, showing (or
refusing) its stability in wide mesh families and extending
the analysis to three-dimensional domains.

Specifically, our main contribution is showing that although the
  inf-sup condition~(\ref{eq:stokes-inf-sup}) does not hold in general
  for the combinations \FEthreeSpaces{\P{1,b}}{\P1}{\P1} and
  \FEthreeSpaces{\P2}{\P1}{\P1} (and also for the symmetric cases
  \FEthreeSpaces{\P1}{\P{1,b}}{\P1} and
  \FEthreeSpaces{\P1}{\P2}{\P1}), it is satisfied in  numerous mesh
  families, namely in \emph{uniformly unstructured}
  meshes (see Definition~\ref{def:uniform-yUnstruct-mesh}). In fact,
  one can easily detect if a given mesh is not unstructured and, with a
  slight post-processing, transform it into a uniformly unstructured one
  (see~Algorithm~\ref{alg:x-unstruct} below).


There are some other well-known unstable elements (for
  instance \FEthreeSpacesP{1}{1}{0} and \FEthreeSpacesP{1,b}{1}{0},
  see Section~\ref{sec:some-appr-with-P0-pressure}) which have less
  degrees of freedom than the ``\textit{mini-element}''
  $\FEthreeSpaces{\P{1,b}}{\P{1,b}}{\P1}$ and that are stable just in
  some concrete meshes. These meshes are, in particular, uniformly
  unstructured and indeed FE combinations provided in this paper are stable in
  a very much larger family of meshes. In this sense,
  $\FEthreeSpaces{\P{1,b}}{\P1}{\P1}$ constitutes a
  ``\textit{minimal-element}'' because  with less degrees
  of freedom than the ``\textit{mini-element}'' stability is provided for unstructured meshes (or  slightly post-processed structured ones).


Previous results are extended to $3D$ domains by using  $\FEfourSpacesP{1,b}{1,b}{1}1$
  and even $\FEfourSpacesP{1,b}111$ (and the corresponding  spaces replacing $\P{1,b}$ by $\P2$). Additionally, we present some numerical tests showing that
$\FEthreeSpacesP{1,b}11$ preserves  the accuracy 
$O(h)$ while \FEthreeSpacesP211 presents a loss of order (only order
$O(h)$) with respect to the order $O(h^2)$ of the classical
Taylor-Hood approximation \FEthreeSpacesP221.

Our interest on the stability of
  anisotropic horizontal/vertical velocity approximation  has been motivated by the study of some  geophysical fluid models governing the motion of large scale oceans (see for instance
\cite{CushmanRoisin-Beckers:09} and references therein cited).  In
fact, they lead to the Navier-Stokes equations with anisotropic eddy
viscosities in shallow domains or, after rescaling, to equivalent
models in adimensional
domains~\cite{Azerad-Guillen:01,Besson-Laydi:92}. In the linear steady
case, the following Anisotropic
Stokes problem arises:
\begin{equation}  \label{eq:AniStokes}
  \left\{
  \begin{aligned}
    -\Delta\uu + \gradx\pp &= \ff \quad \In \domain,
    \\
    -\eps^2\Delta\vv + \dz\pp &= g \quad \In \domain, 
    \\
    \div(\uu,\vv) &= 0 \quad \In\domain,
  \end{aligned}
  \right.
\tag{AS}
\end{equation}
where $\eps>0$ is the so-called aspect ratio between vertical and
horizontal scales.  When $\eps\to 0$, diffusion for the vertical
velocity vanishes \cite{Besson-Laydi:92} and a singular limit
problem appears, which is called the Hydrostatic Stokes equations.
 This type of problems also appear in the transient nonlinear case
  (called Hydrostatic Navier-Stokes equations or Primitive Equations
  of the ocean), see \cite{Azerad-Guillen:01}. 

  In a forthcoming work~\CiteSecondPaper (see also~\cite{Azerad:1994,Azerad:PhD:96})  the well-posedness of the approximations in uniformly unstructured   meshes of~(\ref{eq:AniStokes}) is developed (valid when $\eps$ is
  ``small'' or zero).  This theory is based on
  the Stokes inf-sup condition jointly to an additional inf-sup
  condition needed to stabilize the
  vertical velocity. 
  This new inf-sup condition is satisfied at discrete level when the pressure space
  is large enough with respect to the vertical velocity one. But, for standard  Stokes FE spaces, the
 when the number of degrees of freedom for pressure is too large then Stokes
  stability is not satisfied. The question is, are there intermediate possibilities for approximating velocity-pressure mixed formulation satisfying both stability constraints ? The answer given in this paper is yes, but for uniformly unstructured meshes. 
 
This paper is organized as follows. In
Section~\ref{sec:some-appr-with-P0-pressure}, we recall some results
about approximations with discontinuous pressure. In
Section~\ref{sec:P1bP1P1} we prove that \FEthreeSpacesP{1,b}11 is
stable in uniformly unstructured meshes and show its instability in
some structured meshes. Results are generalized
to the $3D$ case and numerical tests are shown confirming our
analysis. In Section~\ref{sec:P2P1P1}, we show that \FEthreeSpacesP211
is also stable in most unstructured meshes, although the
results are slightly more restrictive than for \FEthreeSpacesP{1,b}11
approximation. This theory is supported by some numerical experiments,
including some $3D$ tests. Finally, in Section~\ref{sec:Q2Q1Q1}, we show instability of the
combination \FEthreeSpaces{\Q2}{\Q1}{\Q1} on quadrangular elements,
both analytical and computationally.

\section{Definitions and notations. \P0 approximation for the pressure}
\label{sec:some-appr-with-P0-pressure}

\renewcommand{\UU}{U}%
\renewcommand{\Uh}{U_h}%
\renewcommand{\uu}{u}%
\renewcommand{\uh}{\uu_{h}}%
\renewcommand{\divx}{\partial_x}%
\newcommand{\divuv}{\partial_x\uu + \partial_y\vv}
\renewcommand{\gradx}{\partial_x}%

Let us assume elsewhere that $\domain\subset\Rset^d$, with $d=2$ or
$d=3$, is a domain with polygonal boundary and let $\Th\subset\Rset^d$
be a shape regular family of triangulations of
$\overline\domain$. Furthermore, let us suppose that $\Th$ is a
simplicial mesh (i.e.~composed of $2D$ triangles or $3D$
tetrahedrons), unless otherwise stated.

In this section we recall some results about discontinuous
approximation of pressure. Specifically, the case $d=2$ is considered
and, for the approximation of the discrete Stokes
problem~(\ref{eq:discreteStokes.a})--(\ref{eq:discreteStokes.b}), the
following \P0 discontinuous space for the approximation of the
pressure is introduced:
\begin{equation}
  \label{eq:P0PhSpace}
  \Ph = \{ \ph\in\Pspace \st \pp|_T\in \P0  \ \forall T\in\Th \}.
\end{equation}
About velocity, we define $\Wh=\Uh\times\Vh$, where $\Uh$ and $\Vh$
are continuous FE spaces for horizontal and vertical velocity,
respectively. Indeed, we deal $\P1$, $\P{1,b}$ or $\P2$ spaces for
horizontal velocity and $\P1$  for vertical
  one. Of course the symmetric choice is equivalent.

As result, we consider the FE combinations denoted by  \FEthreeSpacesP{1}{1}{0},
\FEthreeSpacesP{1b}{1}{0} and \FEthreeSpacesP{2}{1}{0} whose stability
has been sufficiently analyzed in literature  and we can summarize as follows:
\FEthreeSpacesP{2}{1}{0} is the minimal FE combination (with $\P0$
discontinuous pressure) which is stable in general meshes.
Namely, let us consider
\begin{equation}
  \label{eq:P1VhSpace}
  \Vh  = \{ \vh\in H_0^1(\domain) \cap C^0(\overline\domain) \st
  \vv|_T\in \P1, \ \forall T\in\Th \}
\end{equation}

\begin{itemize}
\item {Instability of \FEthreeSpacesP{1}{1}{0} approximation},
  i.e.~the case where $\Uh=\Vh$, is well known to occur  excepting some specific meshes (called \textit{criscross} meshes,
  see~\cite{QinZhang:07}). They are particular cases of \xUnstructured
  and \yUnstructured meshes, see
  Definition~\ref{def:uniform-y-unstructured-parition}.
\item {Instability of \FEthreeSpacesP{1,b}{1}{0}}, i.e.~the case where
  \begin{equation}
    \label{eq:P1bUhSpace}
    \Uh = \{\uh \in H_0^1(\domain)\cap C^0(\overline\domain) \st \uh|_T \in \P{1,b}(T),
    \ \forall T\in\Th
    \},
  \end{equation}
  can be easily reduced to the above
    case, using the fact that integrals in triangles of 
  derivatives of bubble functions vanish. The case
  \FEthreeSpacesP{1,b}{1,b}{0} is similar.
\item Stability of \FEthreeSpacesP{2}{1}{0}, i.e.~the case where 
  \begin{align}
  \label{eq:P2UhSpace}
  \Uh &= \{\uh \in H_0^1(\domain)\cap C^0(\overline\domain) \st \uh|_T \in \P2,
  \ \forall T\in\Th
  \}
\end{align}
while $\Vh$ and $\Ph$ are defined respectively in~(\ref{eq:P1VhSpace})
and (\ref{eq:P0PhSpace}), was shown by R. Stenberg
in~\cite{Stenberg:90}, using the  macro-element technique  (which was
developed by himself in that reference and~\cite{Stenberg:84}).
\end{itemize}

\section{Stability of the \FEthreeSpacesP{1,b}{1}{1} approximation}
\label{sec:P1bP1P1}

In what follows, continuous piecewise linear pressures are selected:
\begin{equation}
  \label{eq:P1PhSpace}
  \Ph = \{ \ph\in\Pspace\cap C^0(\overline\domain) \st \pp|_T\in \P1,  \ \forall T\in\Th \}.
\end{equation}
Using the macro-element technique, we study the 2D case, where
$\Uh$ is a continuous \P{1,b} FE space, as defined
in~(\ref{eq:P1bUhSpace}) and $\Vh$ is a continuous \P1 space, as
in~(\ref{eq:P1VhSpace}). Results are generalized to
the $3D$ case (Section~\ref{sec:P1bP1P1-3D}) and some
numerical experiments are shown, supporting these results
(Section~\ref{sec:numer-simulations}).

\subsection{The 2D case}
\label{sec:P1bP1-P1-2d}

For the study of the \FEthreeSpacesP{1,b}{1}{1} approximation, the
macro-element technique is applied, following the seminal papers
of Stenberg~\cite{Stenberg:84,Stenberg:90} (specially the
former). We show a local stability result
(Theorem~\ref{theorem:P1bP1P1:macroelem}) and global stability result
(Theorem~\ref{theorem:P1bP1P1-mesh-stability}), where a generalization
of Stenberg's theory must be introduced.

A macro-element partitioning, $\Mh$, of a triangulation $\Th$ in $\domain$
is a family of connected sets $M\in\Mh$ which are union of at least
two elements of $\Th$. For each $M\in \Mh$, let $\UM$ and $\VM$ be the
spaces of functions in $H_0^1(M)\cap C^0(\overline M)$ which are,
respectively, in $\P{1,b}$ and $\P1$ for all $T\subset M$. Let $P_M$ be
the space of functions in $C^0(\overline M)$ which are in
$\P1$ for every $T\subset M$.

\begin{definition}
  We say that 
  $\FEthreeSpaces{\Uh}{\Vh}{\Ph}$ is regular in a macro-element $M$  if
  \begin{equation}
    \begin{split}
      \label{eq:macroelement.condition}
      &\text{The set } N_M=\left\{ \ph\in P_M,\ \int_M
        (\dx\uh+\dy\vh) \ph=0 \ \forall (\uh,\vh)\in \UM\times \VM
      \right\}
      \\
      &\text{only consists of the functions which are constant on M}.
    \end{split}
  \end{equation}
  In other case, $\FEthreeSpaces{\Uh}{\Vh}{\Ph}$ is said singular
  in $M$.
  \label{def:stable.finite.elements}
\end{definition}

Stenberg shown, roughly speaking, that if this regularity condition is
satisfied, independently of the geometrical shape of the macro-elements
of some given partitioning (and under slight topological restrictions
on macro-elements, formulated in
Theorem~\ref{theorem:P1bP1P1:macroelem} below) the global inf-sup
condition~(\ref{eq:stokes-inf-sup}) holds uniformly (i.e.~with a
constant $\beta$ independent of the macro-element).

More specifically, according to the theory of Stenberg, we introduce
the following definition:

\begin{definition}
  \label{def:stenberg-equivalence}
  A macro-element $M$ is said to be equivalent (in the sense of Stenberg)
  to a reference macro-element $M^*$ if there exists a continuous
  one-to-one application $F_M : M^*\to M$ which consistently maps
  elements contained in $M^*$ to elements in $M$
  (see~\cite{Stenberg:84} or ~\cite{Stenberg:90} for details).\label{def:3}
\end{definition}
The following result holds (see
Theorem~3.1 in~\cite{Stenberg:90}):
\begin{theorem}
  \label{theorem:stenberg.sufficient.condition}
  Suppose that there exists a family of macro-elements, \Mh, of \Th
  which is composed of a fixed set of equivalence (in the sense of
  Stenberg) classes ${\cal E}_n$ of macro-elements, $n=1,...,N$ and a
  positive integer $L$ ($N$ and $L$ independent of $h$) such that:
  \begin{enumerate}
    \renewcommand{\labelenumi}{(M\arabic{enumi})}
  \item 
    \label{item:M1}
    The condition~(\ref{eq:macroelement.condition}) is satisfied
    for every $M\in \Mh$ (i.e.~\FEthreeSpaces\Uh\Vh\Ph is regular in
    macro-elements).
  \item Each $M\in \Mh$ belongs to one of the classes ${\cal E}_n$,
    $n=1,...,N$.
    \label{item:M2}
  \item If $e$ is an edge (or face) of an element of $\Th$ which is
    interior to $\domain$, then $e$ is interior to at least one
    and no more than $L$ macro-elements of $\Mh$.
    \label{item:M3}
  \end{enumerate}
  Then the Stokes discrete inf-sup
  condition~(\ref{eq:stokes-inf-sup}) holds
  for $\FEthreeSpaces{\Uh}{\Vh}{\Ph}$.
\end{theorem}


In this work we focus on a concrete set of families of macro-elements,
which are defined as union of the elements of \Th which share one only
interior vertex.
More precisely, let us denote by $\interior{\Th}$ the set of all of
the \textit{vertices} of $\Th$ which are in the interior of
$\domain$.

\begin{definition}~
  \begin{itemize}
  \item We say that a macro-element $\Macro$ is centered in a vertex
    $q\in\interior\Th$ if $\Macro$ is the union of every $T\in\Th$
    such that $q$ is a vertex of $T$.

  \item 
    We define a vertex-centered macro-element partitioning of \Th,
    denoted by $\MhOneVertex$, as any macro-element partitioning of \Th
    such that every $\Macro\in\MhOneVertex$ is centered in some vertex
    $q\in\interior\Th$.
  \end{itemize}
%
\end{definition}
If $\Macro$ is centered in $q_0$, we denote by $\Nvertex[q_0]$ (or
just $\Nvertex$) the number of vertices $q$ in $\Macro$ such that
$q\neq q_0$. See that $\Nvertex$ matches also the number of elements
contained in $\Macro$.
For example, Figure~\ref{fig:bubble-macroelement} shows two
macro-elements, each of which is centered in a vertex, $q_0$, with
$\Nvertex=5$ in both cases. 

Shape regularity of $\Th$ implies that
$\exists\,N\in\Nset$ (independent of~$h$) such that each vertex of \Th
is in no more than $N$ elements. On the other hand, if $\Th$ is a
simplicial mesh family in $\Rset^d$, one has $d<\Nvertex[q]$. Therefore:
\begin{equation}
  \label{eq:bound_of_nv}
  d < \Nvertex[q] \le N, \quad \forall q\in\interior\Th.
\end{equation}

Fixed a reference macro-element $\Macro^*$ (centered in a vertex
$q^*\in\interior\Th$) and fixed $n=\Nvertex[q^*]$,
we denote by $\FamilyOneVertex{n}$ the family of macro-elements which
are equivalent (in the sense of Stenberg) to $\Macro^*$. Thus a
macro-element partitioning $\MhOneVertex$ is whatever such that every
$M\in\MhOneVertex$ belongs to some $\FamilyOneVertex{n}$, with
$n\in\Nset$.
Hereafter we assume that \Th satisfies the following slightly
restrictive hypothesis:
\begin{assumption}
  \label{assumption:1}
    Each element $T\in\Th$ has at least one vertex in the interior of,
    $\domain$.
\end{assumption}

  According to~(\ref{eq:bound_of_nv}), for any simplex mesh system
  \Th, one can build at least one and at most $N-d$ families of
  vertex-centered macro-elements, $ \FamilyOneVertex{d+1},\
  \FamilyOneVertex{d+2}\, ...,\ \FamilyOneVertex{N}$. Assumption~\ref{assumption:1}
    ensures that $\Th$ can be covered by macro elements of these
    families.

\begin{lemma}
  \label{lemma:macroelements.with.one.interior.vertex}
  Any vertex-centered macro-element partitioning \MhOneVertex satisfies
  \textit{(M2)} and \textit{(M3)}.
\end{lemma}
\begin{proof}
  Only \textit{(M3)} must be satisfied.
  If $e$ is an edge (or face) in the interior of $\domain$,
  then it is adjacent to, at least, one interior vertex, $q_0$ and
  then $e$ is interior to, at least, the macro-element centered in
  $q_0$.  On the other hand $e$ is interior to (at most) $d$
  macro-elements: that one which are centered in $q_0$ and those ones
  centered in $q_i$, $i=1,...,d-1$, where $q_i$ are the other vertices
  of $e$.
\end{proof}

Thus (in meshes satisfying~Assumption~\ref{assumption:1})
Lemma~\ref{lemma:macroelements.with.one.interior.vertex} and
Theorem~\ref{theorem:stenberg.sufficient.condition} state that
accomplishing the regularity condition~\textit{(M1)} in some
vertex-centered macro-element partitioning $\MhOneVertex$ is sufficient
for inf-sup condition~(\ref{eq:stokes-inf-sup}).
The drawback is that, as it is shown below, not all of 
vertex-centered macro-elements satisfy hypothesis~\textit{(M1)}.  Next
we provide a result characterizing the macro-elements of $\MhOneVertex$
where that hypothesis (in fact~(\ref{eq:macroelement.condition}))
holds for \FEthreeSpacesP{1,b}{1}{1}. The idea is checking the macro-element structure in the following sense.

\begin{definition}
  \label{def:macro.splitting}
  Let $\Th$ be a triangulation of $\domain\subset\Rset^d$ (in
  practice, $d=2$ or $3$) and let $M$ be a macro-element in $\Th$. We
  say that $M$ can be split by an hyperplane $\Pi$ if there exists two
  other macro-elements $M_1$ and $M_2$, composed of elements of $\Th$,
  such that $M_1\cup M_2=M$ and $M_1\cap M_2
  \subset \Pi$.
\end{definition}

\begin{definition}
  \label{def:x-perp-structure}
  A macro-element $M$ of a mesh $\Th$ is be said \xStructured  if $M$ can be split by an hyperplane
  $x=C$, for some $C\in\Rset$. In other case, it is said \xUnstructured.
\end{definition}

\newcommand{\bubbleMacro}{ 
  \begin{tikzpicture}
    \def\nv{5}
    \def\ix{2}\def\iy{1.8}
    \def\vertices{
      \coordinate [label=135:\small $q_0$] (q0) at (0,0);
      \coordinate [label=150:\small $q_1$] (q1) at  (-1.25*\ix,-0.25*\iy);
      \coordinate [label=-100:\small $q_2$] (q2) at  (-0.25*\ix,-\iy);
      \coordinate [label=-45:\small $q_3$] (q3) at  (0.75*\ix,-\iy);
      \coordinate [label=90:\small $q_4$] (q4) at (\ix*1,0.5*\iy);
      \coordinate [label=90:\small $q_5$] (q5) at (-\ix*0.25,\iy);
    }
    \def\showPoints{
      \foreach \i in {0,...,\nv}
      \fill [black] (q\i) circle (3pt);
    }
    \def\showLines{
      \draw (q1) -- (q2) -- (q3) -- (q4) -- (q5) -- cycle;
      \foreach \i in {1,...,\nv} {
        \draw (q\i) -- (q0);
      }
      \foreach \i in {1,...,\nv} {
        \draw [fill, lightgray] (x\i) circle (1.25pt);
        \draw (x\i) circle (2.25pt);
      }      
    }
    \def\showTriangles{
      \node at ($ (q1)!.5!(q2) $) [label=230:\small${T_1}$] {}; 
      \node at ($ (q2)!.5!(q3) $) [label=-90:\small${T_2}$] {}; 
      \node at ($ (q3)!.5!(q4) $) [label=-15:\small${T_3}$] {}; 
      \node at ($ (q4)!.5!(q5) $) [label=90:\small${T_4}$] {}; 
      \node at ($ (q5)!.5!(q1) $) [label=130:\small${T_5}$] {}; 
    }
    \def\xPoints {
      \coordinate [label=215: \small $r_1$] (x1) at ($ (q1)!.5!(q2)!.333!(q0) $);
      \coordinate [label=-90: \small $r_2$] (x2) at ($ (q2)!.5!(q3)!.333!(q0)  $);
      \coordinate [label=-90: \small $r_3$] (x3) at ($ (q3)!.5!(q4)!.333!(q0)  $);
      \coordinate [label=0: \small $r_4$] (x4) at ($ (q4)!.5!(q5)!.333!(q0)  $);
      \coordinate [label=-145: \small $r_5$] (x5) at ($ (q1)!.5!(q5)!.333!(q0)  $);
    }
    \vertices
    \xPoints
    \showPoints
    \showLines
    \showTriangles
  \end{tikzpicture}
}

\newcommand{\unstableBubbleMacro}{ 
  \begin{tikzpicture}
    \def\nv{5}
    \def\ix{2}\def\iy{1.8}
    \def\vertices{
      \coordinate [label=135:\small $q_0$] (q0) at (0,0);
      \coordinate [label=150:\small $q_1$] (q1) at  (-\ix,0*\iy);
      \coordinate [label=-100:\small $q_2$] (q2) at  (-0.25*\ix,-\iy);
      \coordinate [label=-45:\small $q_3$] (q3) at  (0.75*\ix,-\iy);
      \coordinate [label=90:\small $q_4$] (q4) at (\ix*1,0*\iy);
      \coordinate [label=90:\small $q_5$] (q5) at (\ix*0.1,\iy);
    }
    \def\showPoints{
      \foreach \i in {0,...,\nv}
      \fill [black] (q\i) circle (3pt);
    }
    \def\showLines{
      \draw (q1) -- (q2) -- (q3) -- (q4) -- (q5) -- cycle;
      \foreach \i in {1,...,\nv} {
        \draw (q\i) -- (q0);
      }
      \foreach \i in {1,...,\nv} {
        \draw [fill, lightgray] (x\i) circle (1.25pt);
        \draw (x\i) circle (2.25pt);
      }      
    }
    \def\showTriangles{
      \node at ($ (q1)!.5!(q2) $) [label=230:\small${T_1}$] {}; 
      \node at ($ (q2)!.5!(q3) $) [label=-90:\small${T_2}$] {}; 
      \node at ($ (q3)!.5!(q4) $) [label=-15:\small${T_3}$] {}; 
      \node at ($ (q4)!.5!(q5) $) [label=90:\small${T_4}$] {}; 
      \node at ($ (q5)!.5!(q1) $) [label=130:\small${T_5}$] {}; 
    }
    \def\xPoints {
      \coordinate (x1) at ($ (q1)!.5!(q2)!.333!(q0) $);
      \coordinate (x2) at ($ (q2)!.5!(q3)!.333!(q0)  $);
      \coordinate (x3) at ($ (q3)!.5!(q4)!.333!(q0)  $);
      \coordinate (x4) at ($ (q4)!.5!(q5)!.333!(q0)  $);
      \coordinate (x5) at ($ (q1)!.5!(q5)!.333!(q0)  $);
    }
    \vertices
    \xPoints
    \showPoints
    \showLines
    \showTriangles
    \node at ($ (q4)!.4!(q5)!-1!(q0) $) [label=0:$M^+$] {};
    \node at ($ (q4)!.333!(q3)!-0.7!(q0) $) [label=-90:$M^-$] {};
   \end{tikzpicture}
}

\begin{figure}
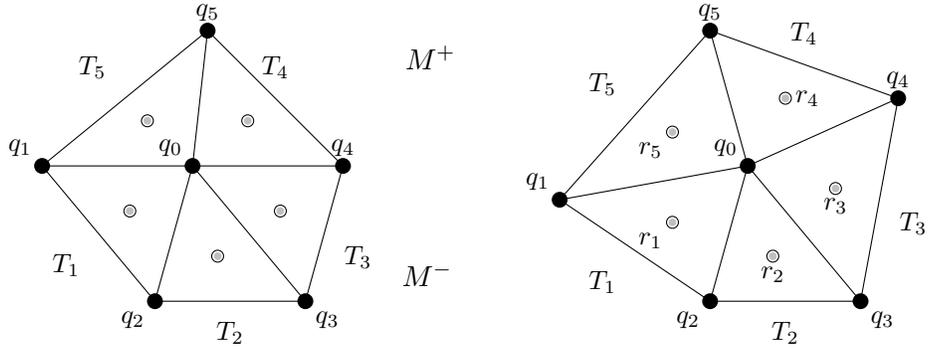

  \centering
  \begin{tabular}{cc}
    \subfloat[Macroelement that can be split by an horizontal line
    (then, it is \yStructured and 
    \FEthreeSpacesP{1,b}{1}{1} is not regular in it).]
    {
      \label{fig:bubble-macroelement.a}
      \unstableBubbleMacro
    }
    &
    \subfloat[Macroelement 
   that  cannot be split by any straight line (it is \xUnstructured and \yUnstructured).] 
    {
      \label{fig:bubble-macroelement.b}
      \bubbleMacro
    }
  \end{tabular}
  \caption{Bubble macroelements.}
  \label{fig:bubble-macroelement}
\end{figure}

Note that a \xStructured macro-element can be split by a straight line
(a plane, in the $3D$ case) which is orthogonal to the $x$ axis.
Similarly, a whole mesh \Th is said \xStructured if it can be
split by an hyperplane $x=C$, $C\in\Rset$. Analogue definitions can be provided
for \yStructured and \zStructured macro-elements or meshes.

For example: 
\begin{itemize}
\item 
Figure~\ref{fig:bubble-macroelement.a} represents a $2D$
macro-element that can be split by an horizontal line (and not by a
vertical line), then it is \yStructured (and \xUnstructured).
\item
The whole mesh in Figure~\ref{fig:structured.mesh} (i.e.~all its
vertex-centered macro-elements) is \xStructured and
\yStructured.
\item
 The macro-element represented in
Figure~\ref{fig:bubble-macroelement.b} cannot be split by any
horizontal or vertical straight line, therefore it is \xUnstructured
and \yUnstructured.
\end{itemize}
  
The following result relates the regularity in a macro-element with its
structure. In fact, it states that, in those macro-elements which are
unstructured in some direction, the bubble functions can be safely
removed from the corresponding component of the velocity field.  
\begin{theorem}   \label{theorem:P1bP1P1:macroelem}
  Let $\Macro$ be a macro-element in a vertex-centered partitioning $\MhOneVertex$. Then:
  \begin{enumerate}
  \item The combination \FEthreeSpacesP{1,b}{1}{1} is regular in
    $\Macro$ (namely~(\ref{eq:macroelement.condition}) holds for
    $\Macro$) if and only if $\Macro$ is \yUnstructured.
  \item The combination \FEthreeSpacesP{1}{1,b}{1} is regular in $\Macro$
    if and only if $\Macro$ is \xUnstructured.
  \end{enumerate}
\end{theorem}
This 
Theorem can also be interpreted as follows: the standard mini-element
\FEthreeSpacesP{1,b}{1,b}1 is ``too rich'' in those macro-elements
which are unstructured in the $x$ or $y$ direction.

\newcommand{\nv}{{n_V}}

\begin{proof}
  We show only the first statement, because the second one is
  symmetric.
 
  \par\noindent\fbox{$\Leftarrow$}
  Let $\Macro\in\MhOneVertex$ be a macro-element with $q_0$ as unique
  interior vertex. Let $\Nvertex=\Nvertex[q_0]$, $q_i$
  ($i=1,...\Nvertex$) the number of elements around $q_0$ and $r_i$
  ($i=1,...,\Nvertex$) the bubble degrees of freedom (barycentres of
  the triangles in $\Macro$). If $\Macro$ is \yUnstructured then there are
  not two vertices $q_i$ and $q_j$ in $\Macro$ horizontally aligned to
  $q_0$.  For each $\ph\in N_\Macro$, that is $ \int_\Macro
  (\dx\uh+\dy\vh) \; \ph =0$, let us denote the shape functions
  $$
  \ph|_{T_i} = a_i + b_i x + c_i y,\quad \forall\,T_i\in
  \Macro.
  $$

  First let us choose $\vh=0$ and $\uh\in \UMacro$ such that $\uh = 1$ on a
  barycentre $r_i$ and $\uh=0$ on all other degrees of freedom in
  $\Macro$ ($r_j$, $j =1,...,\Nvertex$, $j\neq i$ and $q_k$, $k=0,...,\Nvertex$), that
  is $\uh$ is a bubble function on $T_i$. Then,
  \begin{equation*}
    0 = \int_\Macro (\dx\uh+\dy\vh) \; \ph 
    = \int_{T_i} \dx \uh \; \ph 
    = -\int_{T_i} \uh \; \dx \ph 
    = -b_i \int_{T_i} \uh,
  \end{equation*}
  hence
  \begin{equation}
    b_i=0 \quad \forall\, i=1,...,\Nvertex,
    \label{eq:bi=0}    
  \end{equation}
  in particular $\dx\ph=0$ in  $\Macro$.
  Now, imposing the continuity of $\ph$ on $q_0$ and assuming, without loss of generality, $q_0=(0,0)$, it is straightforward
  that there exists $a\in \Rset$ such that
  \begin{equation}
    a_1=...=a_{\Nvertex}=a. 
    \label{eq:ai=0}    
  \end{equation}

  On the other hand, imposing the continuity of $\ph$ in the remaining
  vertices $q_i=(q_{i}^x,q_{i}^y)\in\partial \Macro$ the
  common vertex to $T_{i}$ and $T_{i+1}$ (denoting $T_{\Nvertex+1}=T_1$), one has:
  \begin{equation}
    \label{eq:linear.equations.c_i}
    q_{i}^y(c_{i+1} -  c_i) = 0,
    \quad i=1,...,\Nvertex,
  \end{equation} 
  where (\ref{eq:bi=0}) and~(\ref{eq:ai=0}) have been applied. Therefore
  if $q_i^y\neq 0$ for all $i$, that is if there is no vertex
  horizontally aligned with $q_0$, then there exists $c\in \Rset$ such that
  \begin{equation}
    c_1=...=c_{\Nvertex}=c.
    \label{eq:ci=0}    
  \end{equation}
  This equality is also true when there is only one vertex
  horizontally aligned with $q_0$, because in (\ref{eq:ci=0}) there is one equation plus than unknowns $c_1,...,c_{\Nvertex}$.

  Finally let us choose $\uh=0$ and $\vh\in\VMacro$ defined as $\vh=1$ in $q_0$
  and $\vh=0$ in all other degrees of freedom ($q_i$,
  $i=1,...,\Nvertex$). Applying the quadrature formula $ \int_T f \,dx\;dy
  = \frac{|T|}{3} \sum_{i=1}^3 f(q_i^T)$ with
  $q_i^T$  the vertices of $T$, which is exact in \P1,
  \begin{equation*}
    0 = \int_\Macro \dy\vh \; \ph = - \sum_{i=1}^{\Nvertex} \int_{T_i}\vh \dy\ph
    = -\sum_{i=1}^{\Nvertex} c_i \frac{|T_i|}{3}
  \end{equation*}
  and considering~(\ref{eq:ci=0}) we conclude $c_i=0$ for all
  $i$. Consequently,  $\ph$ is constant on  $\Macro$ and  macro-element
  condition~(\ref{eq:macroelement.condition}) holds.

  \par\noindent\fbox{$\Rightarrow$}
  Assume there are two other vertices horizontally aligned with
  $q_0=(0,0)$ (without loss of generality they are being indexed $q_{j}$
  and $q_{\Nvertex}$). From~(\ref{eq:linear.equations.c_i}) we only deduce:
  \begin{equation*}
    \begin{split}
      c_1&=c_2=...=c_{j}, \\
      c_{j+1}&=c_{j+2}=...=c_{\Nvertex},
    \end{split}
  \end{equation*}
  and in this case, for some $c$, $\widetilde c\in\Rset$,
  \begin{equation}
    \dy\ph(x,y) =
    \begin{cases}
      c \quad \text{ if } y>0, \\
      \widetilde c \quad \text{ if } y\le 0.
    \end{cases}
    \label{eq:ci-two-vertices}  
  \end{equation}
  We are going to use this idea to obtain a counterexample. Indeed, now $\Macro$
  can be split into two disjoint regions: $\Macro^+=\{(x,y)\in \Macro \st y>0\}$
  and $\Macro^-=\{(x,y)\in \Macro \st y<0\}$, each of which can be written as
  union of triangles of $\Macro$ (see
  Figure~\ref{fig:bubble-macroelement}b). For any parameters $a$ and
  $c\in\Rset$, let us define $\ph\in\PMacro$ as follows:
  \begin{equation*}
    \ph(x,y)=
    \begin{cases}
      a-\dfrac{c}{|\Macro^+|} \; y & \text{if } (x,y) \in \Macro^+, \\
      a+\dfrac{c}{|\Macro^-|} \; y & \text{if } (x,y) \in \Macro^-,
    \end{cases}
  \end{equation*}
  where $|R|$ denotes the area of a region $R\subset\Rset^2$. Every
  $\vh\in\VMacro$ can be characterized as 
  the $\P1$ basis function that is equal to $\alpha=\vh(q_0)$ in the
  vertex $q_0$ and $0$ in all other vertices in $\Macro$. Hence one has
  (applying the above mass-lumping quadrature formula):
  \begin{align*}
    \int_{\Macro} \dy \vv_h \; \ph &= -\int_{\Macro} \vv_h \; \dy \ph
    =   
    \frac{c}{|\Macro^+|} \int_{\Macro^+} \vv_h - \frac{c}{|\Macro^-|} \int_{\Macro^-} \vv_h
   \\
    &=
    \frac{c}{|\Macro^+|} \sum_{T\subset {\Macro^+}} \frac{|T|}{3} \alpha
    -
    \frac{c}{|\Macro^-|} \sum_{T\subset {\Macro^-}} \frac{|T|}{3} \alpha
    =\frac{c}{3}\alpha - \frac{c}{3} \alpha
    =0
  \end{align*}

  Then, taking into account that $\dx\ph=0$, the equality
  $\int_\Macro(\dx\uh+\dy\vh)\ph=0$ is satisfied for all $(\uh,\vh)\in
  \UMacro\times \VMacro$, therefore~(\ref{eq:macroelement.condition}) does
  not hold.
\end{proof}

Now, we are interested in extending the local stability result of
Theorem~\ref{theorem:P1bP1P1-mesh-stability} to a global result for
the whole \Th.  According to
Theorem~\ref{theorem:stenberg.sufficient.condition} and
Lemma~\ref{lemma:macroelements.with.one.interior.vertex}, in order to
apply the macro-element theory for the particular family
$\MhOneVertex$, it must be shown that macro-element regularity
condition~(\ref{eq:macroelement.condition}) holds for all
$\Macro\in\MhOneVertex$, namely for every macro-element in the families
$\FamilyOneVertex{d+1},\ ...,\ \FamilyOneVertex{N}$.  The drawback is
that \yStructured macro-elements are part of these families, and
regularity condition is not satisfied for \yStructured macro-elements,
according to Theorem~\ref{theorem:P1bP1P1-mesh-stability} (in the case
of \FEthreeSpacesP{1,b}11).

Note that even if a concrete mesh family, $\Th$, is built
where every $\Macro$ is \yUnstructured (and then regularity
condition~(\ref{eq:macroelement.condition}) is satisfied for every
$\Macro$) one cannot apply
Theorem~\ref{theorem:stenberg.sufficient.condition} to conclude that
inf-sup condition~(\ref{eq:discreteStokes.b}) holds. Indeed, in
Theorem~\ref{theorem:stenberg.sufficient.condition}, $\Mh$ is composed
of \emph{every of the equivalent macro-elements in the sense of
  Stenberg} (Definition~\ref{def:stenberg-equivalence}), independently
of a concrete mesh family (and hence including \yStructured
macro-elements). 
In fact one could build an unstructured mesh family converging to a
structured one, hence stability is not satisfied when $h\to 0$.
Numerical test~\ref{tst:P211-non-unif-unstruct-meshes} shows an
example related to this fact. The solution is consider meshes which
are uniformly unstructured in the following sense.

Let $\MhOneVertex$ be a vertex-centered macro-element partitioning  a
mesh \Th. Let $\Macro\in\MhOneVertex$, let $q$ be its interior vertex
and $\{T_1,...,T_n\}$ the elements of \Th contained in $\Macro$. For
$i=1,...,n-1$, we denote by $\sigma_i$ the angle between the
positive horizontal semiaxis from $q$ and the common edge of the
triangles $T_i$ and $T_{i+1}$ of $\Macro$. Similarly, $\sigma_n$ is
defined using the common edge of $T_n$ and $T_1$.
\begin{definition}
  \label{def:uniform-y-unstructured-parition}
  A vertex-centered macro-element partitioning $\MhOneVertex$  is said uniformly
  \yUnstructured if there exists $\CUnifStruct>0$, independent of $h$,
  such that $|\sin(\sigma_i)| \ge \CUnifStruct$ for all
  $i\in\{1,...,n\}$ except, eventually, one $i_0\in\{1,...,n\}$. A
  uniformly \xUnstructured partitioning is similarly defined (using
  $|\cos(\sigma_i)|$).
\end{definition}
In particular, any uniformly \yUnstructured macro-element partitioning is
composed of \yUnstructured macro-elements and therefore satisfying 
regularity condition~(\ref{eq:macroelement.condition}) for
\FEthreeSpacesP{1,b}11, according to local
Theorem~\ref{theorem:P1bP1P1:macroelem}. Now we extend
Theorem~\ref{theorem:P1bP1P1:macroelem} to a global condition for the
stability of \FEthreeSpacesP{1,b}11 (and
\FEthreeSpacesP1{1,b}1) in a 2D mesh \Th.  Its proof is based on a
relaxation of Stenberg's notion of macro-element equivalence while
strengthening (respect to
Theorem~\ref{theorem:stenberg.sufficient.condition}) the assumptions
on \Th.

\begin{definition}
  \label{def:uniform-yUnstruct-mesh}
  A mesh (family) \Th is said uniformly \yUnstructured
  (\xUnstructured) if it can be covered by a uniformly \yUnstructured 
  (\xUnstructured) vertex-centered macro-element partitioning
  \MhOneVertex.
\end{definition}

\begin{remark}
  \label{rk:xStr-yUnstr-mesh}
  For instance, Figure~\ref{fig:xStr-yUnstr-mesh} shows a mesh \Th
  which is \yUnstructured, i.e.~composed by \yUnstructured
  macro-elements $\Macro\in\MhOneVertex$. In particular one has
  regularity of \FEthreeSpacesP{1,b}11 in every $\Macro$. Indeed,
  $|\sin(\sigma_i)|\ge\sin(\pi/4)>0$ for all $\sigma_i$ in every
  $\Macro\in\MhOneVertex$. If the mesh family \Th preserves this structure
  when $h$ vanishes, it is uniformly \yUnstructured.
  \begin{figure}
    \centering
    \pgfimage[width=0.5\linewidth,height=5\baselineskip]{\imgdir/mesh-xStr-yUnstr}
    \caption{A \xStructured and \yUnstructured mesh, where
      \FEthreeSpacesP{1,b}{1}{1} is Stokes-stable.}
    \label{fig:xStr-yUnstr-mesh}
  \end{figure}
\end{remark}
\begin{theorem}[Global stability  of \FEthreeSpacesP{1,b}11 and
  \FEthreeSpacesP1{1,b}1]
  Let us consider a triangulation $\Th$ of a domain
  $\domain\subset\Rset^2$ satisfying Assumption~\ref{assumption:1}.
  \begin{enumerate}
  \item If \Th is uniformly \yUnstructured, then inf-sup
    condition~(\ref{eq:stokes-inf-sup}) holds for
    \FEthreeSpacesP{1,b}11 elements.
  \item If \Th is uniformly \xUnstructured, then inf-sup
    condition~(\ref{eq:stokes-inf-sup}) holds for
    \FEthreeSpacesP1{1,b}1 elements.
  \end{enumerate}

  \label{theorem:P1bP1P1-mesh-stability}
\end{theorem}

\newcommand{\alphaUnstructured}{%
\CUnifStruct--uniform and \yUnstructured}
\begin{proof}
  Let us focus on the first case. It is sufficient to follow the
  macro-element theory as appears in~\cite{Stenberg:90}, for the
  particular case where \Th is a uniformly \yUnstructured mesh and
  \MhOneVertex is a uniform vertex-centered macro-element partitioning,
  with constant $\CUnifStruct$ independent of $h$.

  The key is in Lemma~3.2 of~\cite{Stenberg:90}, whose proof still
  holds if, instead of equivalent macro-elements (in the sense of
  Stenberg), we consider for each $n=d+1,...,N$ the following
  ``\alphaUnstructured'' equivalence relation (whose equivalence
  classes will be denoted by \FamilyOneVertex{{n,\CUnifStruct}}):
  $\Macro^*$ \text{ and } $\Macro$ are equivalent if the following
  conditions are fulfilled:
  \begin{enumerate}[label=\textit{\alph*)}]
  \item Both $\Macro^*$ and $\Macro$ are in $\FamilyOneVertex{n}$ (in
    particular, they are equivalent in the sense of Stenberg).
  \item $\Macro^*$ and $\Macro$ are uniformly \yUnstructured with constant
    $\CUnifStruct$.
  \end{enumerate}


  Specifically, in Lemma~3.2 of~\cite{Stenberg:90}, let us fix
  $\Family=\FamilyOneVertex{{n,\CUnifStruct}})$.
  Its proof can be followed line by line, under the following
  considerations:
  \begin{enumerate}
  \item Every $M\in \FamilyOneVertex{{n,\CUnifStruct}}$ is
    \yUnstructured, therefore Theorem~\ref{theorem:P1bP1P1:macroelem}
    yields $\beta_M>0$ where 
    $$
    \beta_M = \min_{\stackrel {\ph\in\PM} {\|\ph\|_M=1}} \max_{\stackrel
      {\ph\in\WW_M} {\|\wh\|_{1,M}=1}} (\div\wh,\ph)
    $$
  \item The function $\beta_M=\beta_M(x^1,...,x^k)$, where $x^i$ are
    the vertices of $M$, remains a
    continuous  function.
  \item For any macro-element $M$, let $X(M)$ be the set of coordinates
    of $x^i$ (considered as points $X\in\Rset^{dk}$). Then
    $$
    \{ X(M)\ / \ M\in \FamilyOneVertex{{n,\CUnifStruct}}\}  \subset 
    \{ X(M)\ / \ M\in \FamilyOneVertex{n} \}.
    $$
  \end{enumerate}
  Following the proof of Lemma~3.2 of~\cite{Stenberg:90}, one has that the
  set on the right hand side is compact. Since the set on left is closed,
  one can conclude the proof with
  $\beta_M\ge\beta_{n,\CUnifStruct}>0$. Since $n<N$ and
  $\CUnifStruct$ if fixed, we have an uniform bound.
\end{proof}

\begin{remark}
  \label{rk:mesh-generators-and-error-estimates}
Starting from discrete inf-sup condition 
  (\ref{eq:stokes-inf-sup}), error estimates analysis of mixed
  formulations can be applied \cite{Brezzi-Fortin:91}, obtaining
  \begin{equation}\label{error-estim}
  \|\ww - \ww_h\|_{H^1} + \| p - p_h\|_{L^2} \le C\, h\quad \hbox{and}\quad 
    \|\ww - \ww_h\|_{L^2} \le C\, h^2,
\end{equation}
where $C>0$ depends on the $H^2\times H^1$ norm of $(\ww,p)$.  It is
well-known that second inequality of~(\ref{error-estim}) is based on a
duality argument which requires an assumption on the domain
(regularity or convexity).  
Numerical test~\ref{test:b1-1-error-orders} confirms these statements.
\end{remark}
 
\begin{figure}
    \centering
    \begin{tikzpicture}[>=latex]
      \def\ix{2}\def\iy{1.2} \def\vertices{ \foreach \i in {1,...,5} {
          \foreach \j in {1,...,4} { \coordinate (q\i\j) at
            (\ix*\i,\iy*\j); } } } \def\showLines{
        \draw (q11)--(q51) -- (q54) -- (q14) -- cycle;
        \foreach \i in {1,...,5} { \draw (q\i1) -- (q\i4); } \foreach \j
        in {1,...,4} { \draw (q1\j) -- (q5\j); } \draw (q13) -- (q24);
        \draw (q12) -- (q34); \draw (q11) -- (q44); \draw (q21) --
        (q54); \draw (q31) -- (q53); \draw (q41) -- (q52); \draw [line
        width = 1.1pt] (q21) -- (q31) -- (q42) -- (q43) -- (q33) --
        (q22) -- cycle; } \def\showNodes{ \node (texto) at
        (2.1*\ix,0.25*\iy) {$q_j$}; \node (control1) at
        (2.7*\ix,0.75*\iy) {}; \node (control2) at (2.9*\ix,1.5*\iy) {};
        \draw [black, fill] (q32) circle (1.5pt); \draw [->] (texto)
        .. controls (control1) and (control2) .. (q32); }
      \def\showLayers { \node at ($ (q11)!.5!(q12) $) [left] {\small $
          p(x,z)=- ( y - \frac{1}{6}) \text{ in } L_1 \rightarrow$};
        \node at ($ (q12)!.5!(q13) $) [left] {\small $ p(x,z)= ( y -
          \frac{1}{2}) \text{ in } L_2 \rightarrow$}; \node at ($
        (q13)!.5!(q14) $) [left] {\small $ p(x,z)= -( y - \frac{5}{6})
          \text{ in } L_3 \rightarrow$}; } \def\showCoords { \node at
        (q11) [below] {\small $x_0=0$};

        \node at (q21) [below] {\small $x_1 = \frac 1 4$}; \node at
        (q31) [below] {\small $x_2 = \frac 1 2$}; \node at (q41) [below]
        {\small $x_3 = \frac 3 4$}; \node at (q51) [below] {\small $x_4
          = 1$};

        \node at (q51) [right] {\small $y_0=0$}; \node at (q52) [right]
        {\small $y_1 = \frac 1 3$}; \node at (q53) [right] {\small $y_2
          = \frac 2 3$}; \node at (q54) [right] {\small $y_3 = 1$}; }
      \vertices \showLines \showNodes \showLayers \showCoords
    \end{tikzpicture}
    \caption{A mesh used as counterexample for showing that
      \FEthreeSpacesP{1,b}{1}{1} is not stable in structured meshes}
    \label{fig:structured.mesh}
  \end{figure}
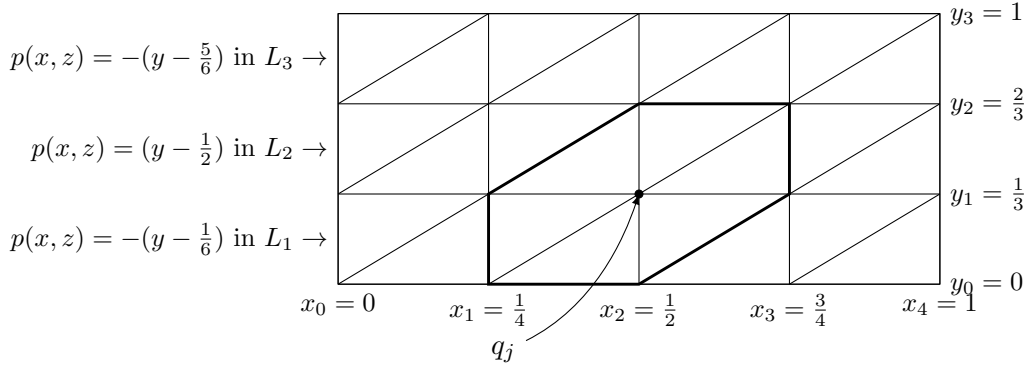

  \begin{example}
    \label{ex:unstructured-mesh}
    It can be demonstrated that in meshes having a strong
    structure (where every $\widehat M\in\MhOneVertex$ is \xStructured
    and \yStructured) LBB condition~(\ref{eq:stokes-inf-sup}) does not
    hold for \FEthreeSpacesP{1,b}11 FE.  For example, let \Th be the
    mesh of Figure~\ref{fig:structured.mesh}. We are going to show
    that there exists a family of non null pressures $\ph\in\Ph$
    satisfying
    \begin{equation}
      \label{eq:p.orthogonal.div(u,v)}
      \int_\Omega \div(u_h,v_h)\ph=0 \quad \forall u_h\in\Uh \text{ and } v_h\in\Vh,
    \end{equation}
    hence the discrete Stokes inf-sup
    condition~(\ref{eq:stokes-inf-sup}) does not hold for
    \FEthreeSpacesP{1,b}11 elements.
  
    Let $x_0,x_1,...,x_k\in\Rset$ and $y_0,y_1,...,y_l\in\Rset$ be two
    uniform partitionings of two generic intervals $[a,b]$, $[c,d]$
    and let $\Th$ be the resulting structured mesh in the rectangle
    $[a,b]\times [c,d]$, with $l$ vertical layers,
    $L_i=(y_{i-1},y_i)$, $i=1,...,l$.  As suggested by the proof of
    Theorem~\ref{theorem:P1bP1P1:macroelem}, let us choose two
    parameters $\widetilde a, \widetilde c\in\Rset$, $\widetilde c\neq
    0$ and consider
  $$
  \ph(x,y) = \widetilde a + (-1)^i \; \widetilde c \; (y-r_i)
  \quad\text{ if $(x,y)\in L_i$}, \quad\text{for some $r_i\in\Rset$}.
  $$
  It is straightforward to test that the condition $\ph\in
  C^0(\overline\domain)$ can be satisfied selecting $r_i=(y_{i-1}+y_i)
  / 2$ and, in this case, if $\widetilde a=0$ then $\ph\in\Ph$ because
  $\int_\domain \ph=0$. Figure~\ref{fig:structured.mesh} shows a
  particular case, with $\widetilde c=1$.
  
  Let $q_j$ and $\phi_j$, $j=1,...,N=(k-1)(l-1)$ be respectively the
  vertices in the interior of $\Th$ and the associated $\P1$ basis
  functions. As $\dx \ph=0$, if $(\uh,\vh)\in \Uh\times\Vh$ then
  $$
  \int_\Omega \div(u_h,v_h)\ph=-\int_\Omega v_h\dy \ph=\sum_{j=1}^N
  a_j \int_\Omega \phi_j \dy \ph
  $$ 
  for some $a_j\in\Rset$. For all $j=1,...,N$, let $T_{ij}$
  ($i=1,\dots ,6$) be the triangles adjacent to the interior vertex
  $q_j$, forming a macro-element like in
  Figure~\ref{fig:structured.mesh}. Without loss of generality, we can
  suppose that $T_{1j}, T_{2j}$ and $T_{3j}$ are laying in an even
  layer, $T_{4j}, T_{5j}$ and $T_{6j}$ in an odd one and $|T_{ij}|=1$
  for all $i$. Then,
  \begin{equation}
    \label{eq:7}
    \int_\Omega \phi_j \dy \ph = \sum_{i=1}^3  \widetilde c \int_{T_{ij}}
    \phi_j -  \sum_{i=4}^6  \widetilde c \int_{T_{ij}}
    \phi_j = \frac{\widetilde c}{3} \left(\sum_{i=1}^3 |T_{ij}| -  \sum_{i=4}^6 |T_{ij}|\right)=0.
  \end{equation}
  Therefore $\int_\Omega\phi_j \dy \ph=0$ for all $j=1,...,N$ and
  (\ref{eq:p.orthogonal.div(u,v)}) holds.
  
  This example is also valid to prove that \FEthreeSpacesP1{1,b}1 is
  not stable in structured meshes.
  \label{rk:example.1}
\end{example}
\begin{remark}
  \label{rk:P1bP1P1-not-stability-in-structured-mesh}
  Theorem~\ref{theorem:P1bP1P1-mesh-stability} provides a sufficient
  but not necessary condition for the (global) stability of 
    \FEthreeSpacesP1{1,b}1 (and \FEthreeSpacesP{1,b}11). 
    Indeed some other (not vertex-centered) macro-element families may
    also conduce to stability.
  %
  \end{remark}
  In order to ensure that \Th is uniformly \xUnstructured (and
  therefore \FEthreeSpacesP1{1,b}1 is stable), the following mesh
  post-processing algorithm may be applied (similarly for
  \FEthreeSpacesP{1,b}11 in \yUnstructured meshes):

    \begin{algorithm}~
      \label{alg:x-unstruct}
      \begin{itemize}
      \item Fix a horizontal ``unstructuring factor''  $r\in(0,1)$ and define $h_r=r\cdot h$, where $h$ is  the mesh size. For
        instance, $r=0.15$ in our experiments.  If $r$ is too small, mesh might not be unstructured enough  and if $r$ is too big, mesh is too deformed. 
      \item For each interior vertex, $q_0$, let $\Macro_{q_0}$ be the macro
        element centered in $q_0$ and let $q_1,...,q_{\Nvertex}$ be
        the other vertices in $\Macro_{q_0}$. We denote $q_i=(q_i^x,q_i^y)$
        its coordinates.
        \begin{itemize}
        \item For each $i=1,...,\Nvertex$, let $d_i=q_i^x - q_0^x$. If
          $|d_i|<h_r$ (i.e.~the edge $q_0q_i$ is ``almost vertical'')
          then:
          \begin{itemize}
          \item For each $j=i+1,...,\Nvertex$, if $|q_j^x-q_0^x| <
            h_r$ (then $\Macro_{q_0}$ is ``almost \xStructured''):
            \begin{itemize}
            \item If $d_i>0$:
              $$q_0^x = q_0^x - (h_r - d_i), \qquad\text{(move $q_0$ to
                the left)}$$
            \item else:
              $$q_0^x = q_0^x + (h_r + d_i),\qquad\text{(move $q_0$ to
                the right).}$$
            \end{itemize}
         \end{itemize}
        \end{itemize}
      \end{itemize}
    \end{algorithm}

    See numerical test~\ref{test:algorithm} for a practical
    experiment.  Note that this algorithm ensures that the resulting
    mesh is uniformly \xUnstructured. Indeed, if for some
    $\Macro_{q_0}$ one has two edges, $q_0q_i$ and $q_0q_j$, ``almost
    vertical'', $q_0$ is displaced and then $|q_0^x-q_i^x|=h_r$. In
    consequence, $\Macro_{q_0}$ and also $\Macro_{q_i}$ (if $q_i$ is an
    interior vertex) are \xUnstructured. Note that, in particular,
    $q_i$ shall not be displaced in following iterations and thus
    $\Macro_{q_0}$ remains \xUnstructured.


\subsection{The $3D$ case}
\label{sec:P1bP1P1-3D}

\renewcommand{\UU}{U}%
\renewcommand{\Uh}{U_h}%
\renewcommand{\VV}{\textbf{V}}%
\renewcommand{\Vh}{\textbf{V}_h}%
\renewcommand{\uu}{u}%
\renewcommand{\uh}{\uu_{h}}%
\renewcommand{\vv}{\textbf{v}}%
\renewcommand{\vh}{\vv_{h}}%
\renewcommand{\divx}{\nabla_x\cdot}%

Results obtained in Section~\ref{sec:P1bP1-P1-2d} are now
generalized to $3D$ domains. Specifically, we aboard the stability of
the FE spaces which arise when one (or two) of the three components of
a $\P1$ continuous approximation of velocity field is enriched with a
bubble per element. Thus Theorem~\ref{theorem:P1bP1P1:macroelem} is
extended to the $3D$ in two ways, depending on whether two or one of
the components are enriched.

\subsubsection*{First possibility: \FEfourSpacesP1{1,b}{1,b}1}
First we focus on the approximation of the discrete Stokes
equations~(\ref{eq:discreteStokes.a})--(\ref{eq:discreteStokes.b}) in
the velocity space $\Wh=\Uh\times\Vh$, being $\Uh$ a \P1 (continuous)
space and $\Vh$ the product of two $\P{1,b}$ (continuous) spaces,
$V_1\times V_2$. Also $\Ph$ is taken as a \P1 continuous space, as
defined in~(\ref{eq:P1PhSpace}).

Let $\Th$ be a $3D$ mesh composed of tetrahedrons, denoted as $K_i$,
satisfying Assumption~\ref{assumption:1} and let $\MhOneVertex$ be the
associated $3D$ vertex-centered macro-element partitioning.  Let $\WM$ be
the space of functions $\wh=(\uh,v_h^1,v_h^2)\in H_0^1(M)^3\cap
C^0(\overline M)^3$ such that $\uh$ is a $\P1$ function in each
$K\subset M$ while $v_h^1,v_h^2$ are $\P{1,b}$ functions and let $\PM$
be the space of $\P1$ continuous functions in each tetrahedron of $M$.
Due to Theorem~\ref{theorem:stenberg.sufficient.condition} and
Lemma~\ref{lemma:macroelements.with.one.interior.vertex}, a sufficient
condition for the stability of \FEfourSpacesP{1}{1,b}{1,b}1 in \Th is
the following $3D$ condition of regularity in macro-elements, for all
$\Macro\in\MhOneVertex$:
\begin{equation}
  \begin{split}
    \label{eq:macroelement.condition.3d}
    &\text{The set } N_\Macro=\left\{ \ph\in P_\Macro \st \int_\Macro \ph \; \div\wh=0, \
      \forall \wh\in \WMacro \right\}
    \\
    &\text{only consists of the functions which are constant on $\Macro$}.
  \end{split}
\end{equation}

Next we are going to prove the following local regularity result,
which constitutes a $3D$ version of
Theorem~\ref{theorem:P1bP1P1:macroelem}.
\begin{theorem}
  \label{theorem:P1P1bP1b-P1:macroelem.3D}
  Let $\Macro$ be a macro-element in a vertex-centered partitioning
  $\MhOneVertex$. Then the combination \FEfourSpacesP{1}{1,b}{1,b}{1}
  is regular in $\widehat M$ if and only if $\widehat M$ is
  \xUnstructured. The cases \FEfourSpacesP{1,b}1{1,b}{1} and
  \FEfourSpacesP{1,b}{1,b}1{1} are symmetric.
\end{theorem}

\begin{proof}\mbox{}
  
  \par\noindent\fbox{$\Leftarrow$}
  Let us define $\ph\in N_\Macro$ with $\ph|_{K_i} = a_i + b_i x + c_i
  y + d_i z$, for each $K_i\subset \Macro$. Let $q_r$,
  $r=1,...\Nvertex$, be the vertices on $\partial \Macro$, let $q_0$
  the only vertex in the interior of $\Macro$ and let $r_i$,
  $i=1,...,n_K$ be the bubble degrees of freedom (where $n_K$ is the
  number of tetrahedrons in $\Macro$).

  If we choose $\uh=0$, $v_h^2=0$ and $v_h^1$ is equal to $1$ in the
  barycentre of a tetrahedron $K_i$ while is $0$ in all other degrees
  of freedom, we have
  $$
  0=\int_\Macro \dy v_h^1 \; \ph = -\int_\Macro v_h^1 \; \dy \ph = -c_i \int_{K_i} v_h^1.
  $$
  Then $c_i=0$ for all $i=1,...,n_K$. Interchanging the roles of
  $v_h^1$ and $v_h^2$, we also deduce $d_i=0$ for all
  $i=1,...,n_K$. Therefore, $\ph|_{K_i}=a_i+b_ix$.

  Now, supposing that $q_0=(0,0,0)$ and applying the continuity of
  $\ph$ in $q_0$, we conclude that exists $a\in\Rset$ such that
  $a_1=...=a_{n_K}=a$.  Then, if $T_{ij}$ denotes the (triangular)
  face common to two tetrahedrons, $K_i$ and $K_j$, the continuity of
  $\ph$ in $T_{ij}$ implies
  \begin{equation*}
    \left\{
  \begin{aligned}
    x_1(b_i-b_j)&=0, \\
    x_2(b_i-b_j)&=0,
  \end{aligned}
  \right.
\end{equation*}
  where $x_1$ and $x_2$ are the respective $x$ coordinates of
  the two vertices which, besides $q_0$, define $T_{ij}$. 

  For every $i$ and $j$ such that the face $T_{ij}$ is not included in
  the plane $x=0$, then $x_1$ or $x_2$ is not zero, hence $b_i=b_j$.
  On the other hand, by hypothesis, there is no plane $x=C$ whose
  intersection with $\Macro$ is formed by faces of tetrahedrons. Then
  any two tetrahedrons $K_i$ and $K_j$ in $\Macro$ can be connected by
  a chain of faces satisfying the property of not inclusion in the
  plane $x=0$, hence $b_i=b_j=b$ for all $i,j=1,...,n_K$, for some
  $b\in\Rset$, and $\ph=a+bx$.

  It remains to prove $b=0$. Let us choose $v_h^1=v_h^2=0$ and
  $\uh$ equal to $1$ in $q_0$ and $0$ in all other
  vertices. Then
  \begin{equation*}
    0 = \int_\Macro \dx\uh \; \ph =  -\int_\Macro \uh \; \dx\ph =
    -b\int_\Macro \uh,
  \end{equation*}
  hence $b=0$, what means $p$ is constant on $\Macro$ and the macro-element
  condition~(\ref{eq:macroelement.condition.3d}) holds.

  \medskip\par\noindent\fbox{$\Rightarrow$} For the reciprocal
  implication, we use a reduction to absurdity argument. Let
  $\Macro\in\MhOneVertex$ with $q_0=(0,0,0)$ its only interior vertex
  and let us suppose that $\Macro$ can be split by the plane $x=0$.
  
  Let $\Macro^+=\{(x,y,z)\in \Macro \st x>0\}$ and $\Macro^-=\{(x,y,z)\in \Macro \st
  x<0\}$ and, for any parameters $a$ and $b\in\Rset$, $b\neq 0$, let
  us consider the family of (not constant) pressures of $\PMacro$:
  $$ 
  \ph(x,y,z)=
  \begin{cases}
    \displaystyle a-\frac b{|\Macro^-|} x & \mbox{if $x\le0$}, \\
    \displaystyle a+\frac b{|\Macro^+|} x & \mbox{if $x>0$}.
  \end{cases}
  $$
  Reasoning like in the proof of Theorem~\ref{theorem:P1bP1P1:macroelem}, it is not
  difficult to show that any $\ph$ in
  this family satisfies $\int_\Macro \ph \div\wh\; =0$ for all $\wh$ in $\WMacro$,
  hence  \FEfourSpacesP{1}{1,b}{1,b}{1} is not regular  in $\Macro$.
\end{proof}

\begin{remark}
  Theorem~\ref{theorem:P1P1bP1b-P1:macroelem.3D} can be applied (owing
  to Theorem~\ref{theorem:stenberg.sufficient.condition} and
  Lemma~\ref{lemma:macroelements.with.one.interior.vertex}) in order
  to provide a sufficient condition for the stability of
  \FEfourSpacesP{1}{1,b}{1,b}{1} in uniformly unstructured $3D$
  meshes. Indeed, a global result similar to
  Theorem~\ref{theorem:P1bP1P1-mesh-stability} can be enounced,
  including also the cases \FEfourSpacesP{1,b}{1}{1,b}{1} and
  \FEfourSpacesP{1,b}{1,b}{1}{1}. 

  Finally, as commented in
  Remark~\ref{rk:mesh-generators-and-error-estimates}, our numerical
  experiments suggest stability in usual unstructured meshes. In any
  case, Algorithm~\ref{alg:x-unstruct} can be generalized and applied
  for ensure stability. Error estimates (\ref{error-estim}) can
  be extended to this $3D$ case.
\end{remark}


\subsubsection*{Second possibility: \FEfourSpacesP11{1,b}1}

\renewcommand{\UU}{\mathbf{U}}%
\renewcommand{\Uh}{\UU_h}%
\renewcommand{\uu}{\mathbf{u}}%
\renewcommand{\uh}{\uu_{h}}%
\renewcommand{\VV}{V}%
\renewcommand{\Vh}{V_h}%
\renewcommand{\vv}{v}%
\renewcommand{\vh}{v_{h}}%
\renewcommand{\divx}[1]{\nabla_\xx\cdot{}#1}
\renewcommand{\gradx}[1]{\nabla_\xx #1}

Next, two components of the velocity field are being left without
enrichment, while the remaining component is enriched with bubbles.
To simplify our exposition, we focus on the case where bubbles are
added ``in the vertical direction'', but not ``in horizontal
directions''. Namely additional bubbles are inserted in the $z$
component of velocity, but not in $x$ or $y$ components, obtaining
\FEfourSpacesP11{1,b}1 FE (the other cases \FEfourSpacesP{1,b}111 and
\FEfourSpacesP1{1,b}11 are symmetric).

Generalizing the 2D case, the idea that we are going to formalize is:
\FEfourSpacesP11{1,b}1 is stable in unstructured in directions $(x,y)$ macro-elements.
 Namely, macro-elements which cannot
be split by planes (more specifically semiplanes, see
Theorem~\ref{theorem:P1P1P1b-P1:macroelem.3D}) with equations
$ax+by+c=0$ (orthogonal to the horizontal plane $z=0$). That is, the
key is: \textit{vertical semiplanes}.

Let $\Th$ be a $3D$ mesh which is composed of tetrahedrons and
satisfies Assumption~\ref{assumption:1}. Let us consider
$\Wh=\Uh\times\Vh$, were $\Uh$ is the product of two $\P{1}$
(continuous) spaces, $U_1$ and $U_2$, while $\Vh$ is a $\P{1,b}$
(continuous) space.  Again, the macro-element regularity
condition~(\ref{eq:macroelement.condition.3d}) is sufficient for the
regularity of \FEfourSpacesP11{1,b}1 in any macro-element
$\Macro\in\MhOneVertex$ where, in this case, $\WMacro$ is the space of
functions $\wh=(u_h^1,u_h^2,\vh)\in H_0^1(\Macro)^3\cap C^0(\Macro)^3$
such that $u_h^1$ and $u_h^2$ are $\P{1}$ while $\vh$ is $\P{1,b}$.
The following geometrical results are going to be useful:
\begin{lemma}
  \label{lemma:vertical.faces}
  Let $\Pi$ a plane containing a triangle, $T$, whose vertices are
  $q_0=(x_0,y_0,z_0)$, $q_1=(x_1,y_1,z_1)$ and
  $q_2=(x_2,y_2,z_2)$. Then $\Pi$ is orthogonal to the plane $z=0$ if
  and only if
  \begin{equation}
    \label{eq:vertical.faces}
    \begin{vmatrix}
      x_1-x_0 & y_1-y_0  \\ x_2-x_0 & y_2-y_0
    \end{vmatrix}=0.
  \end{equation}
\end{lemma}
\begin{proof}
  If $\Pi$ is a plane containing the triangle $T$, then it is defined
  by $q_0$ and the vectors $v_1=(x_1-x_0,y_1-y_0,z_1-z_0)$ and
  $v_1=(x_2-x_0,y_2-y_0,z_2-z_0)$. Then, $\Pi$ is orthogonal to the plane $z=0$
  if and only if the normal vector, $(0,0,1)$, to the plane $z=0$ is a
  linear combination of $v_1$ and $v_2$, that is
  $$
  \begin{vmatrix}
    0 & 0 & 1 \\
    x_1-x_0 & y_1-y_0 & z_1-z_0 \\
    x_1-x_0 & y_2-y_0 & z_2-z_0
  \end{vmatrix}=0,
  $$
  and this equation is equivalent to~(\ref{eq:vertical.faces}).
\end{proof}

\newcommand{\macroChain}[1]{{#1}'}

\begin{lemma}
  \label{lemma.chain.of.faces}
  Let $\Macro\in\MhOneVertex$. Let $\macroChain{M}=K_1\cup...\cup K_n$, where $K_1$,
  ..., $K_n$ are tetrahedrons contained in $\Macro$ which are connected by a chain of
  (triangular) faces which are not orthogonal to the plane $z=0$
  (namely, not vertical faces). If $\ph\in N_\Macro$, there exists
  $a,\macroChain{b},\macroChain{c}\in\Rset$ (with a independent on $\macroChain{M}$) such that
  $p_h|_{\macroChain{M}} = a + \macroChain{b}x + \macroChain{c}y$.
\end{lemma}
\begin{proof}
  Let $\ph\in N_\Macro$, with $p|_{K_i} = a_i + b_ix + c_iy + d_iz$ in each
  tetrahedron $K_i$. We can consider $v_h=1$ in the barycentre of $K_i$
  and $v_h=0$ in all other degrees of freedom, while $u_h^1=u_h^2=0$
  in $\Macro$. Similarly to the proof of
  Theorem~\ref{theorem:P1P1bP1b-P1:macroelem.3D}, we get $d_i=0$ for all
  $i=1,...,n_K$. Supposing, without loss of generality, that the common vertex in $\Macro$
  is $q_0=(0,0,0)$ and using the continuity of $\ph$ in $q_0$ we have
  $a_1=...=a_{n_K}=a$, for some $a\in\Rset$, then
  \begin{equation*}
    p|_{K_i} = a+b_ix+c_iy, \quad \forall i=1,...,n_K.
  \end{equation*}

  If $T_{ij}$ is the common face of two tetrahedrons $K_i$ and $K_j$,
  and $q_1=(x_1,y_1,z_1)$, $q_2=(x_2,y_2,z_2)$ are, besides $q_0$, its
  vertices, the continuity of $\ph$ in $T_{ij}$ is equivalent to
  \begin{equation*}
    \left\{
      \begin{aligned}
        x_1 \; (b_i-b_j) + y_1 \; (c_i-c_j)  &=0, \\
        x_2 \; (b_i-b_j) + y_2 \; (c_i-c_j) &=0.
      \end{aligned}
    \right.
  \end{equation*}
  If $T_{ij}$ is not vertical, from 
  Lemma~\ref{lemma:vertical.faces} (with $x_0=y_0=0$) 
  $$ \begin{vmatrix}
      x_1 & y_1  \\ x_2 & y_2
    \end{vmatrix}\not =0$$
    and then
  $$b_i-b_j=0 \quad\text{ and }\quad c_i-c_j=0. $$

  Finally, if the tetrahedrons $K_1,...,K_n$ can be connected through
  a chain of faces, $T_{1,2}, T_{2,3}$, ..., $T_{n-1,n}$, that are not
  orthogonal to the plane $z=0$, then $b_1=b_{2}=...=b_n=\macroChain{b}$ and
  $c_1=c_{2}=...=c_n=\macroChain{c}$, for some $\macroChain{b}$,
  $\macroChain{c}\in\Rset$.
\end{proof}

For the following result, we use the concept of splitting a
macro-element by a set of semi-planes. This splitting concept can be
defined in the sense of Definition~\ref{def:macro.splitting}.

\begin{theorem}
  Let $\Macro\in\MhOneVertex$ with one only interior vertex, $q_0$. Let $r_0$ be the
  vertical straight line through $q_0$ and let ${\cal F}$ be the
  family of the (vertical) semi-planes bounded by $r_0$.
  \begin{enumerate}
  \item If $\Macro$ cannot be split by semi-planes of ${\cal F}$, then
    \FEfourSpacesP11{1b}1 is regular in $\Macro$.
  \item If $\Macro$ can be split by only  two semi-planes of ${\cal F}$, then
    \FEfourSpacesP11{1b}1 is regular in $\Macro$ if and only if these two
    semi-planes are not aligned (i.e.~they do not form a vertical
    plane splitting $\Macro$).
  \item If $\Macro$ can be split by more than two semi-planes of ${\cal F}$, then \FEfourSpacesP11{1b}1 is \emph{not} regular in $\Macro$.
  \end{enumerate}
  \label{theorem:P1P1P1b-P1:macroelem.3D}
\end{theorem}
\begin{proof} Let $\ph$ in $N_\Macro$ with $\ph|_{T_i}=a_i+b_ix+c_iy+d_iz$, $i=1,...n_K$. Without loss of
  generality, we can assume $q_0=(0,0,0)$. Then, by continuity in $q_0$,
  $a_1=...=a_n$.  Let us choose $u^2_h=v_h=0$ in $\Macro$ while $u^1_h$ is
  equal to $1$ in $q_0$ and $0$ in all other degrees of
  freedom. Using the  quadrature formula defined in the vertices
  of tetrahedrons (exact in \P1), the following condition is satisfied:
  \begin{equation}
    \label{eq:11b1.b_i}
    0 = \int_\Macro \dx u_h^1 \; \ph = -  \sum_{i=1}^{n_K} |K_i| \, b_i.
  \end{equation}
  Applying the same argument to $u^1_h=v_h=0$ while $u^2_h$ is
  equal to $1$ in $q_0$ and $0$ in all other degrees of freedom:
  \begin{equation}
    \label{eq:11b1.c_i}
    0 = \int_\Macro \dy u_h^2 \; \ph =  -\sum_{i=1}^{n_K} |K_i| \,c_i.
  \end{equation}
  
  \paragraph{\textit{1.}}
  Since $\Macro$ cannot be split by semi-planes of ${\cal F}$, then any two
  tetrahedrons of $\Macro$ can be connected by a chain of not vertical
  faces.    Applying Lemma~\ref{lemma.chain.of.faces} to $\macroChain{M}=\Macro$, one has that $\ph$ can be written as $\ph=a+bx+cy$. Then, taking
  into account~(\ref{eq:11b1.b_i}) and~(\ref{eq:11b1.c_i}), $b=0$ and
  $c=0$, hence $p$ is constant on $\Macro$ and the macro-element
  condition~(\ref{eq:macroelement.condition.3d}) holds.

  \paragraph{\textit{2.}} 
  If $\Macro$ can be split only by two semi-planes, $\Pi_1$ and $\Pi_2$ of
  ${\cal F}$ into two regions $\macroChain{M}_1$ and $\macroChain{M}_2$ then (according
  to Lemma~\ref{lemma.chain.of.faces}), there exist
  $b_1,c_1,b_2,c_2\in\Rset$ such that
  \begin{align*}
    \ph|_{\macroChain{M}_1u} &= a + b_1x + c_1y, \\
    \ph|_{\macroChain{M}_2} &= a + b_2x + c_2y.
  \end{align*}
  
  Let $q_0=(0,0,0)$, $q_1=(x_1,y_1,z_1)$ and $q_2=(x_2, y_2, z_1)$
  be the vertices of an interior vertical face $T\subset\Pi_1$. The
  continuity of $\ph$ in $T$ means:
  \begin{align}
    \label{eq:5}
    x_1 (b_1-b_2) + y_1 (c_1-c_2) &= 0,\\
    \label{eq:5b}
    x_2 (b_1-b_2) + y_2 (c_1-c_2) &= 0.
  \end{align}
  But, according to Lemma~\ref{lemma:vertical.faces}, 
  $$
  \begin{vmatrix}
    x_1 & y_1  \\ x_2 & y_2
  \end{vmatrix}=0,
  $$
  then~(\ref{eq:5}) and~(\ref{eq:5b}) are reduced to one only
  equation, e.g.~(\ref{eq:5b}).
  
  Similarly, if $q_0=(0,0,0)$, $\widetilde q_1=(\widetilde x_1,\widetilde
  y_1, \widetilde z_1)$ and $\widetilde q_2=(\widetilde x_2, \widetilde y_2,
  \widetilde z_1)$ are the vertices of an interior vertical face
  $\widetilde T\subset\Pi_2$, one only equation is obtained from the
  continuity of $\ph$ in $\widetilde T$:
  \begin{equation}
    \widetilde x_1 (b_1-b_2) + \widetilde y_2 (c_1-c_2) = 0.
    \label{eq:6}
  \end{equation}
  
   \medskip\par\noindent\fbox{$\Leftarrow$} 
  If the vertical semi-planes $\Pi_1$ and $\Pi_2$ are not aligned then
  the projections of $(x_1,y_1,z_1)\in\Pi_1$ and $(\widetilde
  x_1,\widetilde y_1,\widetilde z_1)\in\Pi_2$  in $z=0$ are not aligned, that is:
  $$ 
  \begin{vmatrix}
    x_1 & y_1  \\ \widetilde x_1 & \widetilde y_2
  \end{vmatrix}\neq 0.
  $$
  Hence, from~(\ref{eq:5b})--(\ref{eq:6}), $b_1=b_2$ and $c_1=c_2$.
  Using~(\ref{eq:11b1.b_i})--(\ref{eq:11b1.c_i}), we conclude
  $b_1=b_2=c_1=c_2=0$, and $\ph=a$ in $\Macro$. 

 \medskip\par\noindent\fbox{$\Rightarrow$} 
  If $\Pi_1$ and $\Pi_2$ are aligned, (\ref{eq:5b})
  and~(\ref{eq:6}) are reduced to one only equation which is not
  enough to conclude $b_1=b_2=c_1=c_2=0$. In fact, counterexamples of
  not constant pressures of $N_\Macro$  satisfying $b_1\neq
  b_2$ or $c_1\neq c_2$, can be obtained similarly to the counterexample given in   the proof of Theorem~\ref{theorem:P1P1bP1b-P1:macroelem.3D}.
  
  \paragraph{\textit{3.}}
  Let $\macroChain{M}_1,...,\macroChain{M}_n$ be regions resulting from splitting $\Macro$ by $n$
  semi-planes ($n>2$), $\Pi_1,...,\Pi_n$ of ${\cal F}$. Arguing as
  in the previous case one has that, for each $i=1,...,n$,
  $\ph|_{\macroChain{M}_i}=a+b_ix+c_iy$ and the continuity of $\ph$ in $\Pi_i$
  can be reduced to one only equation:
  \begin{equation*} 
    x_i (b_i-b_{i+1}) + y_i (c_i-c_{i+1}) = 0,
  \end{equation*}
  where $q_i=(x_i,y_i,z_i)\in\Pi_i$.
Therefore, we have at most $n$ independent equations with $2n$
  unknowns, $b_i$ and $c_i$. Since $n>2$, the
  equations~(\ref{eq:11b1.b_i}) and~(\ref{eq:11b1.c_i}) are not enough
  to conclude $b_i=c_i=0$ for all $i=1,...,n$, because we have at most
  $n+2$ independent equations but $2n$ unknowns (and $n+2<2n$ if
  $n>2$).
\end{proof}

\begin{remark}
  Using Theorem~\ref{theorem:stenberg.sufficient.condition} and
  Lemma~\ref{lemma:macroelements.with.one.interior.vertex}, the
  macro-element regularity result given in
  Theorem~\ref{theorem:P1P1P1b-P1:macroelem.3D} can be used as a
  sufficient condition assuring the stability of \FEfourSpacesP11{1b}1
  in $3D$ meshes that are uniformly \zUnstructured, in the sense of
  meshes which uniformly do not present a local structure like the one
  discussed in Theorem~\ref{theorem:P1P1P1b-P1:macroelem.3D}
  (splitting by vertical semi-planes). In practice, one can hope that
  the unstructured $3D$ meshes obtained with usual software tools
  present this property but, in any case, mesh may be tested and
    eventually slightly modified (via a generalization of
    Algorithm~\ref{alg:x-unstruct}) for  ensuring stability.
\end{remark}

\subsection{Numerical simulations}
\label{sec:numer-simulations}

\renewcommand{\UU}{U}%
\renewcommand{\Uh}{U_h}%
\renewcommand{\uu}{u}%
\renewcommand{\uh}{\uu_{h}}%
\renewcommand{\divx}{\partial_x}%
\renewcommand{\divuv}{\partial_x\uu + \partial_y\vv}
\renewcommand{\gradx}{\partial_x}%

Some computational simulations are provided, which agree with the
numerical analysis of previous sections.

\begin{test}[\FEthreeSpacesP{1,b}{1}{1} and the cavity test in a
  \xStructured and \yUnstructured mesh]
  \label{fig:bubble-x_struct-y_unstruct}
  \begin{figure}
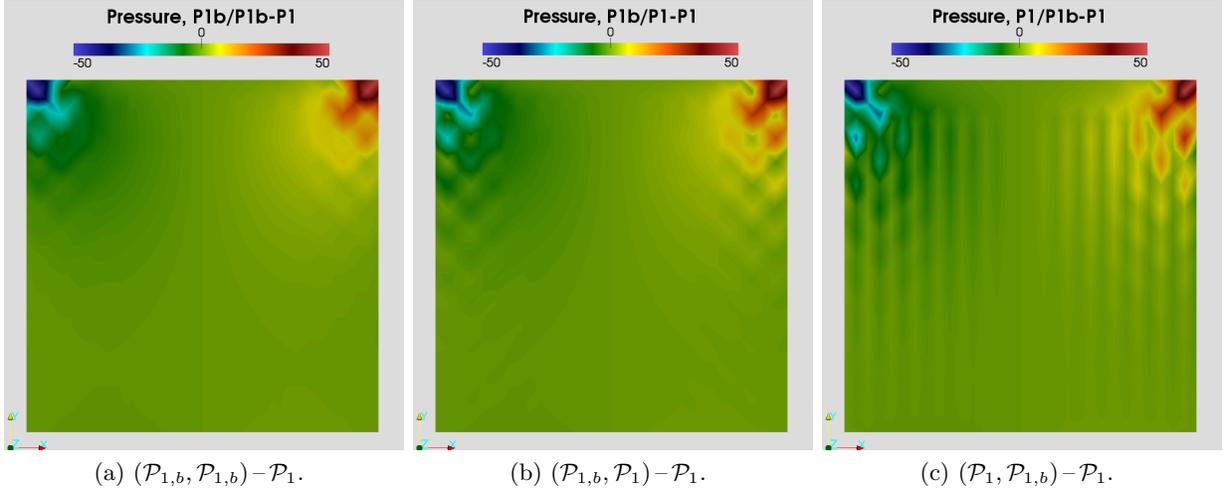

    \centering
    \begin{tabular}{@{}c@{}c@{}c@{}}
      \subfloat[\FEthreeSpacesP{1,b}{1,b}1.]{
        \label{fig:bb-1_x-struc_y_unstruct}
        \pgfimage[width=0.32\linewidth]{\imgdir/bb-1_x-struc_y_unstruct}}
      &
      \subfloat[\FEthreeSpacesP{1,b}{1}1.]{ 
        \label{fig:b1-1_x-struc_y_unstruct}
        \pgfimage[width=0.32\linewidth]{\imgdir/b1-1_x-struc_y_unstruct}}
      &
      \subfloat[\FEthreeSpacesP{1}{1,b}1.]{
        \label{fig:1b-1_x-struc_y_unstruct}
        \pgfimage[width=0.32\linewidth]{\imgdir/1b-1_x-struc_y_unstruct}}
    \end{tabular}
    \caption{Pressure contour plots in a
      \textit{\yUnstructured} mesh (lid-driven cavity test).}
  \end{figure}  

  As first numerical test we have chosen, in the unit square
  $\domain=(0,1)^2$, a \xStructured and \yUnstructured mesh like the
  one shown in Figure~\ref{fig:xStr-yUnstr-mesh}, where the stability
  of \FEthreeSpacesP{1,b}{1}{1} is guaranteed (see
  Remark~\ref{rk:xStr-yUnstr-mesh}).  We have approximated the
  solution of the Stokes
  problem~(\ref{eq:Stokes.a})--(\ref{eq:Stokes.c}) with Dirichlet
  conditions $\uu=1$ on the top boundary, $\uu=0$ on the rest of
  $\partial\domain$ and $\vv=0$ on all $\partial\domain$.

  For the practical implementation, that specific mesh is generated
  (with $15\times 15$ edges, i.e.~$h\simeq 0.666$) by a script
  programmed using the \textit{Python} language, which generates a
  mesh file using the format of the mesh generator
  Gmsh~\cite{Gmsh:09}. This file is imported by the Stokes solver,
  which has been programmed using the \textit{FEniCS} FE
  suite~\cite{FEniCS:LoggMardalEtAl2012a}, taking advantage of the
  facilities offered by this tool for setting the boundary conditions
  on this specific mesh.

  The program has been run for three choices:
  \FEthreeSpacesP{1,b}{1,b}{1}, \FEthreeSpacesP{1,b}{1}{1} and
  \FEthreeSpacesP{1}{1,b}{1}. As expected, the last case presents the
  worst behaviour for the pressure, suggesting the instability of
  \FEthreeSpacesP{1}{1,b}{1} in these \xStructured meshes
  (Figure~\ref{fig:1b-1_x-struc_y_unstruct}). On the other hand, the
  results for \FEthreeSpacesP{1}{1,b}{1} are similar to the
  mini-element  \FEthreeSpacesP{1,b}{1,b}{1} (Figures~\ref{fig:b1-1_x-struc_y_unstruct}
  and~\ref{fig:bb-1_x-struc_y_unstruct}, respectively), as expected
  from the stability of \FEthreeSpacesP{1,b}{1}1 in \yUnstructured
  meshes. 

  Notice that in this test (and in the next ones below) the
  Stokes mixed scheme
  (\ref{eq:discreteStokes.a})-(\ref{eq:discreteStokes.b}) has been
  slightly altered by introducing a standard pressure penalization in
  the divergence equation:
  \begin{equation*}
    \int_\Omega(\dx\uu+\dy\vv) \overline \pp + \varepsilon \int_\Omega\pp\,\overline\pp = 0 ,
    \quad \forall\,\pp\in\Ph,
  \end{equation*}
  where $\varepsilon$ is a small penalization parameter (in this
  example, we have selected $\varepsilon = 10^{-10})$. With this
  approximation, the condition $\int_\domain p=0$ is satisfied
  implicitly (we obtain $\int_\domain p=-5.12\cdot 10^{-8}$) and hence
  it is no longer necessary to include it in the definition of $\Ph$. Then,
  we can take $\Ph=\{\ph\in C^0(\overline\domain) \st \pp|_T\in
  \P1,\quad \forall T\in\Th \}$.
\end{test}

\begin{test}[\FEthreeSpacesP{1,b}{1}{1} and the cavity test]
  \label{test:b11-cavity-test}


  \begin{figure}
    \centering
    \begin{tabular}{@{}c@{}c@{}c@{}}
      \subfloat[Unstructured mesh.]{ 
        \label{fig:P1bP1P1-unstruct-mesh}
        \pgfimage[width=0.32\linewidth]{\imgdir/unstructured-mesh}}
      &
      \subfloat[Velocity field.]{
        \label{fig:b1-1-v-uns}
        \pgfimage[width=0.31\linewidth]{\imgdir/b1-1-v-uns}}
      \subfloat[Pressure.]{
        \label{fig:b1-1-p-uns}
        \pgfimage[width=0.31\linewidth]{\imgdir/b1-1-p-uns}}
    \end{tabular}
    \caption{\FEthreeSpacesP{1,b}{1}{1} elements in a
      \emph{unstructured} mesh.}
  \end{figure}  
  

  \begin{figure}
    \centering
    \begin{tabular}{@{}c@{}c@{}c@{}}
      \subfloat[Structured mesh.]{ 
        \label{fig:P1bP1P1-struct-mesh}
        \pgfimage[width=0.33\linewidth]{\imgdir/structured-mesh}}
      &
      \subfloat[Velocity field.]{
        \label{fig:b1-1-v-str}
        \pgfimage[width=0.31\linewidth]{\imgdir/b1-1-v-str}}
      &
      \subfloat[Pressure.]{
        \label{fig:b1-1-p-str}
        \pgfimage[width=0.31\linewidth]{\imgdir/b1-1-p-str}}
    \end{tabular}
    \caption{\FEthreeSpacesP{1,b}{1}{1} elements in a
      \emph{structured} mesh.}
  \end{figure}

  As the second numerical test, the 2D Stokes
  problem~(\ref{eq:Stokes.a})--(\ref{eq:Stokes.c}) has been solved in
  the squared domain $\domain=(0,1)^2\subset\Rset^2$, using
  \FEthreeSpacesP{1,b}11 FE in two different meshes: an
  unstructured mesh with 3,792 triangles (built with $20$ edges on
  each boundary segment, see Figure~\ref{fig:P1bP1P1-unstruct-mesh})
  and a structured mesh with the same number of boundary edges
  (Figure~\ref{fig:P1bP1P1-struct-mesh}). Then, in both cases
  $h\simeq 1/20$.

  We have fixed the Neumann condition
  $\nu\frac{\partial\uu}{\partial
    z}=1$ on the top boundary, while
  $\uu=0$ on the rest of $\partial\domain$ and $\vv=0$ on
  all $\partial\domain$.
  In this case, the implementation
  has been developed in \textit{FreeFem++}~\cite{FreeFem++}.

  In both simulations (unstructured and structured mesh) the velocity
  filed present similar behavior (Figures~\ref{fig:b1-1-v-uns}
  and~\ref{fig:b1-1-v-str}). But the pressure unknown, that is correct
  in the unstructured case (Figure~\ref{fig:b1-1-p-uns}), presents
  clear non-physical oscillations in the unstructured one
  (Figure~\ref{fig:b1-1-p-str}) which has a somewhat ``checker-board''
  structure similar to classical unstable approximations of Stokes
  problem~\cite{girault-raviart:86}.  Hence this experiment confirm
  that \FEthreeSpacesP{1,b}11 is stable in meshes which cannot be
  split by any horizontal segment, as predicted by the $2D$ numerical
  analysis developed in Section~\ref{sec:P1bP1-P1-2d}.
\end{test}

\begin{test}[\FEthreeSpacesP{1,b}{1}{1} and error orders]
  \label{test:b1-1-error-orders}
 
  \begin{figure}
    \centering
    \begin{tabular}{@{}c@{}c@{}}
      \subfloat
      {
        \pgfimage[width=0.49\textwidth]{\imgdir/b1-1-errors-loglog}
        \label{fig:b1-1-errors-loglog}
      }
      &
      \subfloat
       {
        \pgfimage[width=0.49\textwidth]{\imgdir/bb-1-errors-loglog}
        \label{fig:bb-1-errors-loglog}
      }
    \end{tabular}
    \caption{Velocity and pressure errors for \FEthreeSpacesP{1,b}11 (left) and
      \FEthreeSpacesP{1,b}{1,b}1 (right)}
  \end{figure}
  Thirdly, to give an estimate of the convergence orders,
  we have considered a 2D example in $\Omega=(0,1)^2$
  whose exact solution is: 
  \begin{align} 
    \uu(x,y) &= \cos(2\*\pi\*x)\*\sin(2\*\pi\*y) -  \sin(2\*\pi\*y),
    \quad
    \vv(x,y) = -\uu(y,x),
    \label{eq:exact-sol.velocity}
    \\
    \pp(x,y) &=  2\*\pi\*(\cos(2\*\pi\*y)-\cos(2\*\pi\*x)).
    \label{eq:exact-sol.pressure}
  \end{align}
  Note that $\dx\uu+\dy\vv=0$, $(\uu,\vv)|_{\partial\domain}=0$ and
  $\int_\domain \pp=0$.  The external force $\ff$, is calculated so that
  the momentum equation~(\ref{eq:Stokes.a}) hold.

  The problem has been solved in some unstructured
  meshes (recall that the theory predicts  instability of \FEthreeSpacesP{1,b}11
  in structured meshes), with $h\simeq 2^{-2}$, $2^{-3}$, ..., $2^{-8}$.
  The absolute errors obtained have been plotted in
  Figure~\ref{fig:b1-1-errors-loglog}, while
  Table~\ref{tab:b1-1-orders} shows the values of the convergence
  orders associated to those errors, which are calculated as
  $\log(e_{h_2}/e_{h_1})/\log(h_2/h_1)$, when $h_1<h_2$ travel through
  the different mesh sizes.

Taking into account error estimates (\ref{error-estim}), we can state the following conclusions:
  \begin{itemize}
  \item For both components of the velocity field, optimal convergence
    is obtained, that is order $O(h^2)$ in $L^2(\Omega)$ and order
    $O(h)$ in $H^1(\Omega)$.  For pressure, optimal order $O(h)$ in
    $L^2(\Omega)$ is obtained.
  \item The results are very similar (almost identical for the
    velocity components) to the classical mini-element,
    \FEthreeSpacesP{1,b}{1,b}1, whose errors are shown in
    Figure~\ref{fig:bb-1-errors-loglog}. It is significant to note
    that the combination \FEthreeSpacesP{1,b}11 preserves the error
    orders while requires a minor number of degrees of freedom than \FEthreeSpacesP{1,b}{1,b}1 and
      a smaller computational effort. On the other
    hand, to approach the Anisotropic Stokes  problem (see (AS) above), \FEthreeSpacesP{1,b}11 is stable in  most meshes, while \FEthreeSpacesP{1,b}{1,b}1 looses its stability
    (cf.~\CiteSecondPaper).
  \end{itemize}

  \begin{table}
    \centering
    \small
    \begin{tabular}{l@{\qquad}r@{\qquad}cccccc}
      \\ \toprule
      & $h$
      & $2^{-3}$ & $2^{-4}$ & $2^{-5}$ & $2^{-6}$ & $2^{-7}$  & $2^{-8}$
      \\ \otoprule
      \multirow{2}*{$\uu$} 
      &  $\|u-u_h\|_{L^2}$ &  
      1.099  & 1.709  & 1.952  &  1.876  & 1.988 & 2.019 
      \\ \cmidrule{2-8}
      & $\|u-u_h\|_{H^1_0}$ & 
      1.155  & 1.053  & 1.092  & 0.990  & 1.023  & 1.041 
      \\ \midrule
      \multirow{2}*{$\vv$} 
      & $\|v-v_h\|_{L^2}$ & 
      1.914  & 2.076  & 2.071  & 1.981  & 2.059  & 2.058
      \\ \cmidrule{2-8}
      & $\|v-v_h\|_{H^1_0}$ & 
      0.912  & 1.046  & 1.010  & 0.994  & 1.022  & 1.022
      \\ \midrule
      {$\pp$}
      & $\|p-p_h\|_{L^2}$ &  
      0.558  & 1.817  & 0.808  & 1.414  & 0.809 & 0.988
      \\ \bottomrule
    \end{tabular}
    \caption{Error orders for velocities and pressure
      (\FEthreeSpacesP{1,b}11 FE)}
    \label{tab:b1-1-orders}
  \end{table}

\end{test}

\begin{test}[Post-precessing Algorithm~\ref{alg:x-unstruct}]
  \label{test:algorithm}
  We have used Algorithm~\ref{alg:x-unstruct} (with $r=0.15$) for
  modifying a strongly structured mesh (defined by $16\times 16$
  intervals in the unit square, see Figure~\ref{fig:algo-mesh1}),
  arriving at the unstructured mesh given in Figure~\ref{fig:algo-mesh2}. Pressure for a
  standard cavity-driven test shows instabilities for
  \FEthreeSpacesP{1,b}11 in the original structured mesh
  (Figure~\ref{fig:algo-mesh1}) but not in the post-processed one
  (Figure~\ref{fig:algo-mesh2}).
    \begin{figure}
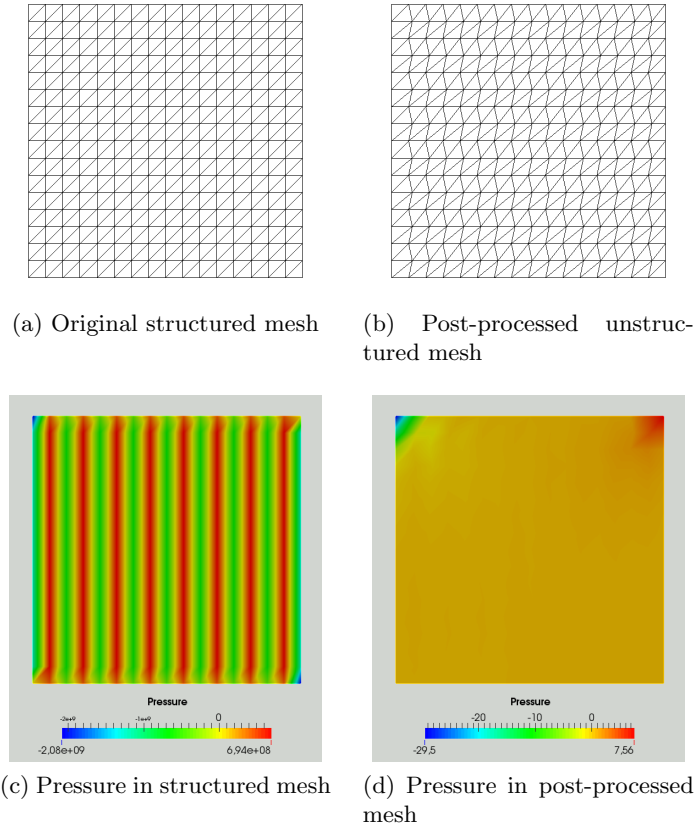

    \centering
    \begin{tabular}{ccc}
      \subfloat[Original structured mesh]{ 
        \pgfimage[width=0.25\linewidth]{\imgdir/algo-mesh-16x16-struc}
        \label{fig:algo-mesh1}}
      &
      \subfloat[Post-processed unstructured mesh]{ 
        \pgfimage[width=0.25\linewidth]{\imgdir/algo-mesh-16x16-unstr}
        \label{fig:algo-mesh2}}
      \\ 
      \subfloat[Pressure in structured mesh]{ 
        \pgfimage[width=0.25\linewidth]{\imgdir/algo-pressure-1b-1-struc}
        \label{fig:algo-pressure1}}
      &
      \subfloat[Pressure in post-processed mesh]{ 
        \pgfimage[width=0.25\linewidth]{\imgdir/algo-pressure-1b-1-unstr}
        \label{fig:algo-pressure2}}
    \end{tabular}
    \caption{Test of Algorithm~\ref{alg:x-unstruct} in a structured mesh.}
    \label{fig:P1bP1P1-strucured-cavity-test-3d}
  \end{figure}  
\end{test}

\begin{test}[3D domain, \FEfourSpacesP{1,b}{1,b}{1}{1} or
  \FEfourSpacesP{1,b}{1}{1}{1} and cavity test]
  \label{test:cavity-bb1-1-and-b11-1}

  \begin{figure}
    \centering
    \begin{tabular}{ccc}
      \subfloat[\FEfourSpacesP{1,b}{1,b}{1,b}1]{ 
        \pgfimage[width=0.3\linewidth]{\imgdir/bbb-1-str}
        \label{fig:P1bP1P1-strucured-cavity-test-3d-a}}
      &
      \subfloat[\FEfourSpacesP{1,b}{1,b}{1}1]{ 
        \pgfimage[width=0.3\linewidth]{\imgdir/bb1-1-str}
        \label{fig:P1bP1P1-strucured-cavity-test-3d-b}}
      &
      \subfloat[\FEfourSpacesP{1,b}{1}{1}1]{ 
        \pgfimage[width=0.3\linewidth]{\imgdir/b11-1-str}
        \label{fig:P1bP1P1-strucured-cavity-test-3d-c}}
    \end{tabular}
    \caption{Comparison of bubble $3D$ FE in
      a \emph{structured mesh}.}
    \label{fig:P1bP1P1-strucured-cavity-test-3d}
  \end{figure}  

  \begin{figure}
    \centering
    \begin{tabular}{ccc}
      \subfloat[\FEfourSpacesP{1,b}{1,b}{1,b}1]{ 
        \pgfimage[width=0.3\linewidth]{\imgdir/bbb-1-uns}
        \label{fig:P1bP1P1-unstructured-cavity-test-3d-a}}
      &
      \subfloat[\FEfourSpacesP{1,b}{1,b}{1}1]{ 
        \pgfimage[width=0.3\linewidth]{\imgdir/bb1-1-uns}
        \label{fig:P1bP1P1-unstructured-cavity-test-3d-b}}
      &
      \subfloat[\FEfourSpacesP{1,b}{1}{1}1]{ 
        \pgfimage[width=0.3\linewidth]{\imgdir/b11-1-uns}
        \label{fig:P1bP1P1-unstructured-cavity-test-3d-c}}
    \end{tabular}
    \caption{Comparison of bubble $3D$ FE in
      an \emph{unstructured mesh}.}
    \label{fig:P1bP1P1-unstructured-cavity-test-3d}
  \end{figure}  

  In the $3D$ case, we have approximated the solution of the Stokes
  equations~(\ref{eq:Stokes.a})--(\ref{eq:Stokes.c}) in the cube domain  $\domain=(0,1)^3\subset\Rset^3$, using both:
  \begin{enumerate}
  \item A structured mesh, constructed by
    \texttt{FreeFem++}~\cite{FreeFem++} through the subdivision in
    three tetrahedrons of each ones of the cubes resulting of the
    division of $\domain$ into $32^3$ equal rectangular
    parallelepipeds.
  \item a unstructured mesh consisting of 190,968 tetrahedrons,
    constructed by the mesh generator \texttt{Gmsh}~\cite{Gmsh:09} and
    then imported to \texttt{FreeFem++}.
  \end{enumerate}

  The Dirichlet condition $u(x,y,z)=y(y-1)$, $v(x,y,z)=w(x,y,z)=0$ has
  been fixed on the top boundary, while $\ww=0$ has been fixed on the rest of
  $\partial\domain$ (where $\ww=(u,v,w)$ denotes the
   solution). The FE approximations
  \FEfourSpacesP{1,b}{1,b}{1,b}1, \FEfourSpacesP{1,b}{1,b}11 and
  \FEfourSpacesP{1,b}111 have been used in each one of the two meshes
  and the resulting surface plots of the approximated pressure, $\ph$,
  is shown in Figure~\ref{fig:P1bP1P1-strucured-cavity-test-3d}
  (structured mesh) and Figure~\ref{fig:P1bP1P1-unstructured-cavity-test-3d}
  (unstructured one).

  Owing to the choice of a color map that highlights the small
  absolute values of $\ph$, spurious oscillations of the pressure are
  evident for \FEfourSpacesP{1,b}{1,b}11 and \FEfourSpacesP{1,b}111 in
  the structured mesh
  (Figures~\ref{fig:P1bP1P1-strucured-cavity-test-3d-b}
  and~\ref{fig:P1bP1P1-strucured-cavity-test-3d-c}), where we have
  stated their instability. On the other hand,
  Figures~\ref{fig:P1bP1P1-unstructured-cavity-test-3d-b}
  and~\ref{fig:P1bP1P1-unstructured-cavity-test-3d-c} suggest a
    correct behaviour (excepting some small oscillations on
  the top of the domain) if the mesh is unstructured, confirming
  theoretical results in Section~\ref{sec:P1bP1P1-3D}.
\end{test}

\begin{test}[$3D$ domain, enriching with bubble in the correct direction]
  In this test it has been considered a mesh $\Th$ of the cubic $\Omega=(0,1)^3$, which presents a
  clear structure only in the $z$-direction. Namely, \Th is
  \zStructured while \xUnstructured and \yUnstructured.

  In concrete, we have considered at the top of the domain, $S$, an
  unstructured triangulation ${\Th}_2$ of $S$ (defined with 64
  sub-intervals on $\partial S$, i.e.~$h\simeq 0.015$). Then a $3D$ 
  mesh, $\Th$, has been constructed, defining it from the extension of
  ${\Th}_2$ along $16$ layers uniformly distributed in the $z$-direction.

  We have programed the lid driven cavity test (as detailed in
  previous tests) for \FEfourSpacesP{1,b}{1,b}{1,b}1,
  \FEfourSpacesP{1}{1}{1,b}1 and \FEfourSpacesP{1,b}111 FE
  combinations. In the resulting graphics for pressure, shown in
  Figure~\ref{fig:z-structured-cavity-test-3d-bis}, it can be seen
  that the two first cases
  (Figures~\ref{fig:z-structured-cavity-test-3d-bis-a}
  and~\ref{fig:z-structured-cavity-test-3d-bis-b}) present a
  similar correct qualitative behaviour, while the third one
  (Figure~\ref{fig:z-structured-cavity-test-3d-bis-c}) presents
  some non physical oscillations (which, in this graphics, are
  amplified by the selection of the color map).

  This behaviour agree with the theory developed in
  Section~\ref{sec:P1bP1P1-3D}, which suggest the stability of
  \FEfourSpacesP11{1,b}1. Indeed, for this mesh, we can hope that the
  macro-elements $\Macro\in\MhOneVertex$, are \xUnstructured and
  \yUnstructured, hence adding bubble functions to the first
  components of the velocity field is not mandatory. But these
  macro-elements are \zStructured and then, for obtaining regularity,
  it is necessary and sufficient adding bubble functions to the third
  component. Also due to this fact, we cannot hope stability for
  \FEfourSpacesP{1,b}111, (see
  Figure~\ref{fig:z-structured-cavity-test-3d-bis-c}).

  \begin{figure}
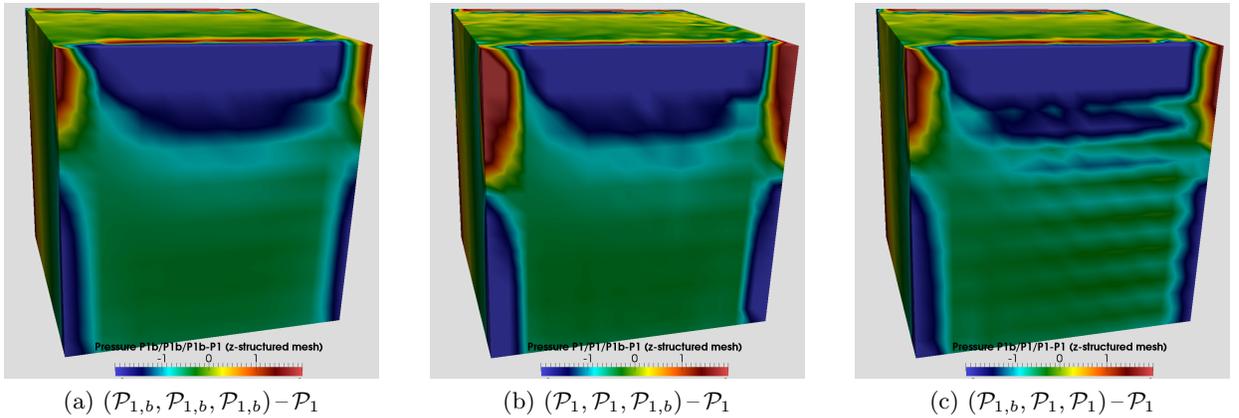

    \centering
    \begin{tabular}{ccc}
      \subfloat[\FEfourSpacesP{1,b}{1,b}{1,b}1]{ 
        \pgfimage[width=0.3\linewidth]{\imgdir/bbb-1-zstr}
        \label{fig:z-structured-cavity-test-3d-bis-a}}
      &
      \subfloat[\FEfourSpacesP{1}{1}{1,b}1]{ 
        \pgfimage[width=0.3\linewidth]{\imgdir/11b-1-zstr}
        \label{fig:z-structured-cavity-test-3d-bis-b}}
      &
      \subfloat[\FEfourSpacesP{1,b}{1}{1}1]{ 
        \pgfimage[width=0.3\linewidth]{\imgdir/b11-1-zstr}
        \label{fig:z-structured-cavity-test-3d-bis-c}}
    \end{tabular}
    \caption{Comparison of bubble $3D$ FE in
      a \emph{$z$--structured mesh}.}
    \label{fig:z-structured-cavity-test-3d-bis}
  \end{figure}  

\end{test}

\section{About stability of $\FEthreeSpacesP{2}{1}{1}$ FE}
\label{sec:P2P1P1}

\subsection{The 2D case}
\label{sec:P2P1-P1-2d}

In this section we introduce a $\P2$ continuous space for the
approximation of the horizontal component of velocity field,
\begin{equation*}
  \Uh = \{\uh \in H_0^1(\domain)\cap C^0(\domain) \st \uh|_T \in \P2,
  \ \forall T\in\Th
  \},
\end{equation*}
while $\Vh$ and $\Ph$ are approximated by $\P1$ continuous elements,
as defined in (\ref{eq:P1VhSpace}) and (\ref{eq:P1PhSpace}).

As in Section~\ref{sec:P1bP1-P1-2d}, let us suppose \Th satisfying
Assumption~\ref{assumption:1} 
and let \MhOneVertex be the a vertex-centered macro-element partitioning
of \Th, as in Figure~\ref{fig:macroelement:1}.  Let $\UM$ be the space
of functions in $H_0^1(M)\cap C^0(\overline M)$ which are $\P{2}$ in
each element $T\subset M$ and let $\VM$ and $\PM$ as in
Section~\ref{sec:P1bP1-P1-2d}.  In this framework,
Lemma~\ref{lemma:macroelements.with.one.interior.vertex} means that a
sufficient condition for the discrete inf-sup
condition~(\ref{eq:stokes-inf-sup}) is that $\MhOneVertex$ satisfies
(\ref{eq:macroelement.condition}), namely \FEthreeSpacesP{2}{1}{1} is
regular in every $\Macro\in\MhOneVertex$. The following result
characterizes this condition in function of the number of elements
contained in $\Macro$ and the number of vertices which are
horizontally aligned with $q_0$. See that this case is a little more
technical than the similar one
(Theorem~\ref{theorem:P1bP1P1:macroelem}) presented in
Section~\ref{sec:P1bP1-P1-2d} (see
Remark~\ref{rk:macro211-vs-macro-b11} for further details).

\begin{theorem}
  \label{theorem:P2P1P1:macroelem}
  Let $\Macro\in\MhOneVertex$ be a macro-element which can be written as
  union of $\Nvertex$ elements which share exactly one interior common
  vertex, $q_0$.
  \begin{enumerate}
  \item If $\Macro$ is \yStructured (i.e.~there are two vertices
    horizontally aligned to $q_0$), then \FEthreeSpacesP{2}{1}{1} is
    not regular in $\Macro$, namely
    condition~(\ref{eq:macroelement.condition}) does not hold.
  \item If there is just one vertex in $\Macro$ horizontally
    aligned with $q_0$, then \FEthreeSpacesP{2}{1}{1} is regular in $M$.
  \item If there is no vertex horizontally aligned with
    $q_0$:
    \begin{enumerate}
    \item If $\Nvertex$ is odd, then \FEthreeSpacesP{2}{1}{1} is regular in $\Macro$.
    \item If $\Nvertex$ is even, then \FEthreeSpacesP{2}{1}{1} is regular in $\Macro$
      if and only if the following algebraic condition does \emph{not}
      hold:
      \begin{equation}
        \label{eq:instabilityMacroCondition}
        \sum_{i=1}^{\Nvertex} (-1)^{i} \cot(\sigma_i)\left(\frac1{|T_i|}+ \frac1{|T_{i+1}|} \right) = 0,
      \end{equation}
    \end{enumerate}
    where $\cot(\sigma_i)=\cos(\sigma_i)/\sin(\sigma_i)$ is the
    cotangent function and $\sigma_i$ are the angles introduced before
    Definition~\ref{def:uniform-y-unstructured-parition}
  \end{enumerate}
\end{theorem}
\begin{remark}
  The case \FEthreeSpacesP121 is analogue and, for brevity, is not
  enounced here.
\end{remark}
\newcommand{\macroelemA}[2]{
  \def\yCoordA{#1}
  \def\yCoordB{#2}
  \begin{tikzpicture}
    \def\Nvertex{5}
    \def\ix{2.7}\def\iy{1.8}
    \def\vertices{
      \coordinate [label=-135:$q_0$] (q0) at (0,0);
      \coordinate [label=above: $q_1$] (q1) at (\ix*0.5,\iy);
      \coordinate [label=above: $q_2$] (q2) at (-\ix*0.5,\iy);
      \coordinate [label=135: $q_3$] (q3) at  (-\ix,\yCoordB*\iy);
      \coordinate [label=-45: $q_4$] (q4) at  (0.5*\ix,-\iy);
      \coordinate [label=45: $q_5$] (q5) at (\ix*1,\yCoordA*\iy);
    }
    \def\xPoints {
      \coordinate [label=135: $s_1$] (x1) at ($ (q1)!.5!(q0) $);
      \coordinate [label=215: $s_2$] (x2) at ($ (q2)!.5!(q0) $);
      \coordinate [label=-90: $s_3$] (x3) at ($ (q3)!.5!(q0) $);
      \coordinate [label=0: $s_4$] (x4) at ($ (q4)!.5!(q0) $);
      \coordinate [label=above: $s_5$] (x5) at ($ (q5)!.5!(q0) $);
    }
    \def\showPoints{
      \foreach \i in {0,...,\Nvertex}
      \fill [black] (q\i) circle (3pt);
      \foreach \i in {1,...,\Nvertex} {
        \draw [fill, lightgray] (x\i) circle (2.25pt);
        \draw (x\i) circle (2.25pt);
      }
    }
    \def\showLines{
      \draw (q1)--(q2) -- (q3) -- (q4) -- (q5) -- cycle;
      \foreach \i in {1,...,\Nvertex} {
        \draw (q\i) -- (q0);
      }
    }
    \def\showTriangles{
      \node at ($ (q5)!.5!(q1) $) [label=+45:${T_1}$] {}; 
      \node at ($ (q1)!.5!(q2) $) [label=above:${T_2}$] {}; 
      \node at ($ (q2)!.5!(q3) $) [label=135:${T_3}$] {}; 
      \node at ($ (q3)!.5!(q4) $) [label=below:${T_4}$] {}; 
      \node at ($ (q4)!.5!(q5) $) [label=below:${T_5}$] {}; 
    }
    \vertices
    \xPoints
    \showPoints
    \showLines
    \showTriangles
  \end{tikzpicture}
}

\begin{figure}
  \centering 
  {\macroelemA{-0.3}{-0.5}}
  \caption{\FEthreeSpacesP{2}{1}{1} macro-element with a unique interior
    vertex, $q_0$, and $\Nvertex=5$.}
  \label{fig:macroelement:1}
\end{figure}

\begin{proof} We split the proof into four steps.

  ~\step{Algebraic characterization of $N_\Macro$.}

  Let $T_i$, $i=1,...,\Nvertex$, be the elements of $\Macro$. If $\ph\in N_\Macro$, let us
  define the form functions $ \ph|_{T_i} = a_i + b_i x + c_i y$.
  As in Figure~\ref{fig:macroelement:1}, let us denote by $q_0$ the only
  vertex interior to $\Macro$, $q_i$, $i=1,...\Nvertex$, the vertices in $\partial
  \Macro$ and $s_i$, $i=1,...,\Nvertex$ the midpoints of the edges in the interior
  of $\Macro$.

  Choosing $\vh=0$ and $\uh\in
  \UMacro$ defined as $\uh = 1$ on a midpoint $s_i$ and $\uh=0$ on all other
  degrees of freedom of $\Macro$ ($s_j$, $j =1,...,\Nvertex$, $j\neq i$ and $q_k$,
  $k=0,...,\Nvertex$), then
  \begin{align*}
    0 = -\int_\Macro \div(\uh,\vh)\;\ph &= -\int_\Macro \dx \uh \; \ph = 
    \int_\Macro \uh \; \dx \ph =
    \\
    &= b_i \int_{T_i} \uh + b_{i+1} \int_{T_{i+1}} \uh 
    = \frac{|T_j|}{3} b_j + \frac{|T_{j+1}|}{3} b_{j+1},
  \end{align*}
  where we have applied the quadrature formula $ \int_T f =
  \frac{|T|}{3} \sum_{i=1}^3 f(s_i^T), $ which is exact on \P2 if $s_i^T$
  are the midpoints of the edges of $T$.  Again, we have identified
  the index $i=\Nvertex+1$ with $i=1$. Defining $\alpha_{i}=|T_{i}|$, the
  following linear system for $b_i$ is obtained:
  \begin{equation}
    \begin{pmatrix}
      \alpha_1 & \alpha_2 & 0 &  \dots & 0 & 0\\
      0 & \alpha_2 & \alpha_3 &  \dots & 0 & 0\\
      \vdots& \vdots & \vdots & \ddots  & \vdots & \vdots \\
      0 & 0 & 0 & \dots & \alpha_{\Nvertex-1} & \alpha_{\Nvertex} \\
      \alpha_1 & 0 & 0 & \dots & 0 & \alpha_{\Nvertex} \\
    \end{pmatrix}.
    \begin{pmatrix}
      b_1 \\ b_2 \\ \vdots  \\ b_{\Nvertex-1} \\ b_{\Nvertex}
    \end{pmatrix}
    =
    \begin{pmatrix}
      0 \\ 0 \\ \vdots \\ 0 \\ 0
    \end{pmatrix}
    \label{eq:P2P2P1.linear.system}
  \end{equation}
  A simple calculation shows that, if $b_1=-b/\alpha_1$ then $b_2=b/\alpha_2$, $b_3=-b/\alpha_3$, $\dots$, $b_{\Nvertex}=(-1)^{\Nvertex} b/\alpha_{\Nvertex}$ and $b_1=(-1)^{\Nvertex+1}b/\alpha_1$.
  Therefore,  the solution
  of~(\ref{eq:P2P2P1.linear.system}) can be characterized as follows:
  \begin{equation}
    \label{eq:b_i:oddeven}
    \left\{
      \begin{aligned}
        &\text{If $\Nvertex$ is \textit{odd}}, \quad b_i=0,
        \quad \forall i=1,...,\Nvertex. \\
        &\text{If $\Nvertex$ is \textit{even}}, \quad
        b_i=\frac{(-1)^{i}}{\alpha_i} b, \quad \forall i=1,...,\Nvertex,
        \quad \forall  b\in\Rset.
      \end{aligned}
    \right.
  \end{equation}

  Taking 
  $\uh=1$ in $q_0$, $\uh=0$ on all other degrees of freedom in $\Macro$ and
  $\vh=0$ does not provide any new information, because the midpoint
  quadrature formula means that $\int_{T_i} \uh =0$ for all $i$.

  Finally let us choose $\uh=0$ and $\vh\in \VMacro$ defined as $\vh = 1$ on the
  interior vertex, $q_0$, and of course $\vh=0$ on all vertices on
  $\partial \Macro$. Then $0  = (\vh,\dy \ph) = \sum_{i=1}^{\Nvertex} c_i \int_{T_i} \vh = \sum_{i=1}^{\Nvertex} c_i
  |T_i|/{3}$ hence
  \begin{equation}\label{eq:sumAlphaiCi}
    \sum_{i=1}^{\Nvertex}\alpha_i c_i = 0.
  \end{equation}

  Now assuming, without loss of generality, $q_0=(0,0)$ and imposing
  the continuity of $\ph$ on $q_0$, it is straightforward that
  $$a_1=...=a_{\Nvertex}=:a.$$
   If $q_i=(q_i^x,q_i^y)\in\partial \Macro$, $i=1,...,\Nvertex$, being $q_0$ and $q_i$ the   common vertices to $T_{i}$ and $T_{i+1}$
  (Figure~\ref{fig:macroelement:1}), continuity of $\ph$ in $q_i$ means:
  \begin{equation}
    \label{eq:linear.equations.c_i:b_i}
    q_i^y (c_{i+1} -  c_i) = q_i^x (b_{i+1} - b_i),
    \quad \forall\, i=1,...,\Nvertex.
  \end{equation}

  \step{Proof of case 2.}

  If there exists a unique vertex on $\partial \Macro$, without loss
  of generality it can be named $q_1$, such that $q_1^y=0$ (that is,
  $q_1$ is horizontally aligned with $q_0$), the
  equation for $i=1$ in~(\ref{eq:linear.equations.c_i:b_i}) means $b_{2}-b_1=0$ (note that $q_1^x\not=0$).
  According to~(\ref{eq:b_i:oddeven}), this does not give any new
  information in the case of $\Nvertex$ odd (in which $b_i=0$ for all $i$)
  but in the even case, we obtain
  $$
  \left(\frac{(-1)^{2}}{\alpha_{2}}  - \frac{(-1)^{1}}{\alpha_{1}}\right)b=0,
  $$
  that is $b=0$ and then $b_i=0$ for $i=1,...,\Nvertex$.  Now from
  (\ref{eq:linear.equations.c_i:b_i}) there exists $c\in \Rset$ such
  that $c_2=c_3= ... = c_{\Nvertex}=c_1=c$. Then, from (\ref{eq:sumAlphaiCi}) $c=0$ (and   $c_i=0$ for all $i$), therefore $\ph=a$ and the macro-element
  condition~(\ref{eq:macroelement.condition}) holds.

  \step{Proof of case 3.}

  Now let us suppose that there is not any vertex horizontally
  aligned with $q_0=(0,0)$, that is $q_i^y \neq 0$ for all
  $i=1,...,\Nvertex$. According to~(\ref{eq:b_i:oddeven}), the
  linear equations~(\ref{eq:linear.equations.c_i:b_i}) can be written as
  \begin{align*}
    q_i^y (c_{i+1} -  c_i) &=
    q_i^x \left(\frac{(-1)^{i+1}}{\alpha_{i+1}}-\frac{(-1)^{i}}{\alpha_{i}}\right)b
    \\
    &=
    q_i^x (-1)^{i+1}\left(\frac{1}{\alpha_{i+1}}+\frac{1}{\alpha_i}\right)
    b, \quad \forall\,i=1,...,\Nvertex,
  \end{align*}
  with $b=0$ if $\Nvertex$ is odd. We can divide by $q_i^y$, obtaining
  \begin{equation}
    \label{eq:12}
    (c_{i+1}-c_i) + (-1)^{i} \widehat m_i \,\widehat \alpha_i \, b = 0,
     \quad \forall\,i=1,...,\Nvertex,
  \end{equation}
  where $\widehat m_i= q_i^x/q_i^y$ is the inverse of the slope of
  the segment $\overline{q_0 q_i}$, hence $\widehat
  m_i=\cos(\sigma_i)/\sin(\sigma_i)=\cot(\sigma_i)$, and
  $\widehat\alpha_i = 1/\alpha_{i}+1/\alpha_{i+1}=1/|T_i| +
  1/|T_{i+1}|$. Therefore,
  $$
  \widehat m_i \,\widehat \alpha_i=\cot(\sigma_i)\left(\frac1{|T_i|}+ \frac1{|T_{i+1}|} \right).
  $$

  If $\Nvertex$ is \textit{odd} (case \textit{3.a}), $b=0$ and (\ref{eq:sumAlphaiCi}) and 
  (\ref{eq:12}) imply $c_i=0$ for all $i=1,...,\Nvertex$, concluding that
  $p_h=a$ and~(\ref{eq:macroelement.condition}) holds.
  
  If $\Nvertex$ is \textit{even} (case \textit{3.b}), let us
  write~(\ref{eq:12}) and (\ref{eq:sumAlphaiCi}) as the algebraic system:
  \begin{equation}
    \begin{pmatrix}
      -1 & 1 & 0 &  \dots & 0 & 0 &  -\widehat m_1 \widehat\alpha_1\\
      0 & -1 & 1 &  \dots & 0 & 0 &  \widehat m_2 \widehat\alpha_2\\
      \vdots& \vdots & \vdots & \ddots  & \vdots & \vdots &\vdots \\
      0 & 0 & 0 & \dots & -1 & 1 &  - \widehat m_{\Nvertex-1} \widehat\alpha_{\Nvertex-1} \\
      1 & 0 & 0 & \dots & 0 & -1 &  \widehat m_{\Nvertex}
      \widehat\alpha_{\Nvertex}\\
      \alpha_1 & \alpha_2 & \alpha_3 & \dots & \alpha_{\Nvertex-1} & \alpha_{\Nvertex} & 0
    \end{pmatrix}.
    \begin{pmatrix}
      c_1 \\ c_2 \\ \vdots  \\c_{\Nvertex-1} \\ c_{\Nvertex} \\ b
    \end{pmatrix}
    =
    \begin{pmatrix}
      0 \\ 0 \\ \vdots \\ 0 \\ 0 \\ 0
    \end{pmatrix}.
    \label{eq:P2P2P1.linear.system.full}
  \end{equation}
  Adding the first $\Nvertex$ rows of this system, we obtain the equation $ S \* b =
  0 $, where
  $$S=\sum_{i=1}^{\Nvertex} (-1)^{i}\widehat m_i \widehat\alpha_i
  = \sum_{i=1}^{\Nvertex} (-1)^{i} \cot(\sigma_i)\left(\frac1{|T_i|}+ \frac1{|T_{i+1}|} \right).
  $$
  
 If we assume $S\neq 0$, then $b=0$ and the first $\Nvertex$ equations
  of~(\ref{eq:P2P2P1.linear.system.full}) can be written as
  $$
  -c_1+c_2=0,\quad -c_2+c_3=0,\ ...,\quad -c_{\Nvertex}+c_1=0,
  $$ 
  hence $c_1=c_2=\cdots=c_{\Nvertex}\equiv c \in\Rset$. But then, last equation
  of~(\ref{eq:P2P2P1.linear.system.full}) means c $\sum_{i=1}^{\Nvertex}
  \alpha_i =0$, which implies $c=0$.  In short, if $S\neq 0$ the only
  solution of the system is $c_1=\cdots=c_{\Nvertex}=b=0$ and therefore $p_h
  = a$.

  Reciprocally, if $S=0$ the first $\Nvertex$ rows of $A$ are linearly
  dependent and the system~(\ref{eq:P2P2P1.linear.system.full}) has got
  non trivial solutions $(c_1,..., c_{\Nvertex}, b)\neq 0$. For those values, any
  pressure defined by 
  $$
  \ph|_{T_i}=a+\frac{(-1)^{i}}{\alpha_i} b\,x+c_i \,y
  $$
  satisfies (\ref{eq:P2P2P1.linear.system})--(\ref{eq:linear.equations.c_i:b_i}), therefore it
  is a non-constant element of $N_\Macro$
  and~(\ref{eq:macroelement.condition}) does not hold.

  \step{Proof of case 1.}

  Finally, we show that if there exist  two vertices
  $q_j$ and $q_k$ on $\partial \Macro$ horizontally aligned with $q_0$ ($1\le j<k \le \Nvertex$) then
  $\Macro$ is not regular. Without loss of generality, renumbering the
  vertices if necessary, we can suppose $1<j<k=\Nvertex$. Then equations $j$ and
  $\Nvertex$ in~(\ref{eq:linear.equations.c_i:b_i}) imply respectively
  $b_{j+1}=b_j$ and $b_1=b_{\Nvertex}$, which (using~(\ref{eq:b_i:oddeven}))
  yields $b=0$ (in both cases) and then $b_i=0$ for all $i=1,...,\Nvertex$.

Taking this into account in the rest of equations of (\ref{eq:linear.equations.c_i:b_i}), we have
  \begin{align*}
    c_2-c_1 =0,\ ..., c_{j}-c_{j-1}=0 \quad &\text{and} 
    \quad c_{j+2}-c_{j+1}=0,\ ...,\ c_{\Nvertex} - c_{\Nvertex-1}=0
    \\
    \intertext{and then, there exists $c$ and $\widetilde c\in\Rset$, such that}
    c_1=c_2=\cdots=c_j\equiv c \quad &\text{and} \quad
    c_{j+1}=c_{j+2}=\cdots=c_{\Nvertex}\equiv\widetilde c .
  \end{align*}
  This fact and condition~(\ref{eq:sumAlphaiCi}) imply
  \begin{equation}
    \label{eq:c and widetilde c}
    \widetilde c = \frac{- \sum_{i=1}^j \alpha_i}{\sum_{i=j+1}^{\Nvertex} \alpha_i}
    \cdot c.
  \end{equation}
  Note that $\sum_{i=1}^j \alpha_i$ and $\sum_{i=j+1}^{\Nvertex} \alpha_i$ are the measures of the two macro-elements splitting $\Macro $.  Therefore, it is clear that any pressure with $\ph|_{T_i}=c$ for $1\le i\le j$ and
  $\ph|_{T_i}=\widetilde c$ for $j+1\le i\le \Nvertex$, with $\widetilde c$ given
  by~(\ref{eq:c and widetilde c}), is a non constant element of $N_\Macro$ and
  then~(\ref{eq:macroelement.condition}) does not hold.
\end{proof}
\begin{remark}
  \label{rk:P2P1P1-counterexample-structured-mesh}
  Like in the case of \FEthreeSpacesP{1,b}11 given in
  Section~\ref{sec:P1bP1-P1-2d}, the ideas of
  Theorem~\ref{theorem:P2P1P1:macroelem} can be applied to provide
  counterexamples showing that \FEthreeSpacesP211 is not stable in
  some families of strong structured meshes. In fact, the same example
  given in Remark~\ref{rk:P1bP1P1-not-stability-in-structured-mesh} is
  valid in the current context.
\end{remark}

\begin{remark}
  \label{rk:macro211-vs-macro-b11}
  \newcommand{\macroelemD}{
    \begin{tikzpicture}
      \def\Nvertex{6}
      \def\ix{2.5}\def\iy{1.5}
      \def\vertices{
        \coordinate [label=-45:$q_0$] (q0) at (0,0);
        \coordinate [label=right: $q_1$] (q1) at (\ix*1,0);
        \coordinate [label=+45: $q_2$] (q2) at (\ix,\iy);
        \coordinate [label=+135: $q_3$] (q3) at  (0,\iy);
        \coordinate [label=+135: $q_4$] (q4) at   (-\ix,0);
        \coordinate [label=225: $q_5$] (q5) at (-\ix,-\iy);
        \coordinate [label=-45: $q_6$] (q6) at (0,-\iy);
      }
      \def\xPoints {
        \coordinate (x1) at ($ (q1)!.5!(q0) $);
        \coordinate (x2) at ($ (q2)!.5!(q0) $);
        \coordinate (x3) at ($ (q3)!.5!(q0) $);
        \coordinate (x4) at ($ (q4)!.5!(q0) $);
        \coordinate (x5) at ($ (q5)!.5!(q0) $);
        \coordinate (x6) at ($ (q6)!.5!(q0) $);
      }
      \def\showPoints{
        \foreach \i in {0,...,\Nvertex}
        \fill [black] (q\i) circle (3pt);
        \foreach \i in {1,...,\Nvertex} {
          \draw [fill, lightgray] (x\i) circle (2.25pt);
          \draw (x\i) circle (2.25pt);
        }
      }
      \def\showLines{
        \draw (q1)--(q2) -- (q3) -- (q4) -- (q5) -- (q6) -- cycle;
        \foreach \i in {1,...,\Nvertex} {
          \draw (q\i) -- (q0);
        }
      }
      \def\showTriangles{
        \node at ($ (q6)!.5!(q1) $) [label=-45:${T_6}$] {}; 
        \node at ($ (q1)!.5!(q2) $) [label=0:${T_1}$] {}; 
        \node at ($ (q2)!.5!(q3) $) [label=90:${T_2}$] {}; 
        \node at ($ (q3)!.5!(q4) $) [label=135:${T_3}$] {}; 
        \node at ($ (q4)!.5!(q5) $) [label=180:${T_4}$] {}; 
        \node at ($ (q5)!.5!(q6) $) [label=-45:${T_5}$] {}; 
      }
      \vertices
      \xPoints
      \showPoints
      \showLines
      \showTriangles
    \end{tikzpicture}
  }
  \newcommand{\macroelemE}{
    \begin{tikzpicture}
      \def\Nvertex{6}
      \def\ix{2.2}\def\iy{0.9}
      \def\vertices{
        \coordinate [label=-45:$q_0$] (q0) at (0,0);
        \coordinate [label=right: $q_1$] (q1) at (\ix,\iy);
        \coordinate [label=+45: $q_2$] (q2) at (0,2*\iy);
        \coordinate [label=+135: $q_3$] (q3) at  (-\ix,\iy);
        \coordinate [label=+135: $q_4$] (q4) at   (-\ix,-\iy);
        \coordinate [label=225: $q_5$] (q5) at (0,-2*\iy);
        \coordinate [label=-45: $q_6$] (q6) at (\ix,-\iy);
      }
      \def\xPoints {
        \coordinate (x1) at ($ (q1)!.5!(q0) $);
        \coordinate (x2) at ($ (q2)!.5!(q0) $);
        \coordinate (x3) at ($ (q3)!.5!(q0) $);
        \coordinate (x4) at ($ (q4)!.5!(q0) $);
        \coordinate (x5) at ($ (q5)!.5!(q0) $);
        \coordinate (x6) at ($ (q6)!.5!(q0) $);
      }
      \def\showPoints{
        \foreach \i in {0,...,\Nvertex}
        \fill [black] (q\i) circle (3pt);
        \foreach \i in {1,...,\Nvertex} {
          \draw [fill, lightgray] (x\i) circle (2.25pt);
          \draw (x\i) circle (2.25pt);
        }
      }
      \def\showLines{
        \draw (q1)--(q2) -- (q3) -- (q4) -- (q5) -- (q6) -- cycle;
        \foreach \i in {1,...,\Nvertex} {
          \draw (q\i) -- (q0);
        }
      }
      \def\showTriangles{
        \node at ($ (q6)!.5!(q1) $) [label=-45:${T_1}$] {}; 
        \node at ($ (q1)!.5!(q2) $) [label=0:${T_2}$] {}; 
        \node at ($ (q2)!.5!(q3) $) [label=90:${T_3}$] {}; 
        \node at ($ (q3)!.5!(q4) $) [label=135:${T_4}$] {}; 
        \node at ($ (q4)!.5!(q5) $) [label=180:${T_5}$] {}; 
        \node at ($ (q5)!.5!(q6) $) [label=-45:${T_6}$] {}; 
      }
      \vertices
      \xPoints
      \showPoints
      \showLines
      \showTriangles
    \end{tikzpicture}
  }

  \begin{figure}
    \centering
    \begin{tabular}{cc}
      \subfloat[Two vertices
      horizontally aligned with $q_0$.]{ 
        \label{fig:macro.notStable.A}
        \macroelemD }
      &
      \subfloat[No vertex horizontally aligned, but $n$ is even
      and~(\ref{eq:instabilityMacroCondition}) holds.]{
        \label{fig:macro.notStable.B}
        \macroelemE }
    \end{tabular}
    \caption{Two macro-elements where $\FEthreeSpacesP{2}{1}{1}$ is not regular.}
    \label{fig:macro.notStable}
  \end{figure}  
  \label{rk:not-stability-of-P2P1P1}
  According to Theorem~\ref{theorem:P2P1P1:macroelem}, there are two
  situations where \FEthreeSpacesP{2}{1}{1} may be singular in a
  macro-element with one only interior vertex $q_0$:
  \begin{enumerate}
  \item[\textit{(a)}] either there exist two vertices horizontally aligned with
    $q_0$,
  \item[\textit{(b)}] or $\Nvertex$ is even and the algebraic
    condition~(\ref{eq:instabilityMacroCondition}) holds.
  \end{enumerate}

  The case \textit{(a)}, which is illustrated in
  Figure~\ref{fig:macro.notStable.A}, is coincident to the singularity
  condition studied for \FEthreeSpacesP{1,b}11 in
  Section~\ref{sec:P1bP1P1}, but the case \textit{(b)} is specific of
  \FEthreeSpacesP211 and does not appears in
  the~\FEthreeSpacesP{1,b}11 context. In this sense, the regularity of
  \FEthreeSpacesP211 is more restrictive than the regularity of
  \FEthreeSpacesP{1,b}11.

  The case \textit{(a)} is specially relevant because constitutes the
  basic brick in many vertical structured meshes, like the one in
  Figure~\ref{fig:structured.mesh} (used in the counterexample
  presented in
  Remark~\ref{rk:P2P1P1-counterexample-structured-mesh}). Also this is
  the kind of mesh used in the numerical experiment that produced the
  instability given in Figure~\ref{fig:P2P1P1-p-str}.

  About condition \textit{(b)}, in practice it is difficult to find
  random macro-elements of this kind, i.e.~macro-elements for which
  $\Nvertex$ is even, no vertex is horizontally aligned with $q_0$
  and~(\ref{eq:instabilityMacroCondition}) is satisfied. One of these
  cases (of what can be called ``weakly \yStructured macro-elements'')
  is presented in Figure~\ref{fig:macro.notStable.B}.  Indeed,
  assuming $|T_i|=|T_{i+1}|$ for $i=1,...,6$, then $\cot(\sigma_i)=1$
  for $i=1,4$, $\cot(\sigma_i)=-1$ for $i=3,6$ and $\cot(\sigma_i)=0$
  for $i=2,5$, hence~(\ref{eq:instabilityMacroCondition}) holds.

  But notice that, in the case of the macro-element in
  Figure~\ref{fig:macro.notStable.B}, the continuity of the cotangent
  function implies that slight modifications of the triangles, conduce
  to elements where~(\ref{eq:instabilityMacroCondition}) is not
  satisfied, i.e.~regular macro-elements.  In any case, for discrete
  inf-sup condition~(\ref{eq:stokes-inf-sup}),
  equation~(\ref{eq:instabilityMacroCondition}) must be uniformly
  satisfied, i.e.~(\ref{eq:stokes-inf-sup}) does not hold (with
  constant independent on $h$) if the left hand side in
  equation~(\ref{eq:instabilityMacroCondition}) tends to zero.

  On the other hand, it is easy to find macro-element families
  where \FEthreeSpacesP211 is regular.  For example, in the case
  $\Nvertex=4$, if each vertex $q_i$, $i=1,...,4$ is in a different
  quadrant (Figure~\ref{fig:regularQuadrantVertex.A}) then the
  singularity condition~(\ref{eq:instabilityMacroCondition}) does not
  hold. In effect, if we denote $Q_i$ the $i$-th quadrant and suppose
  $q_i\in Q_i$ for all $i=1,...,4$ then $\cot(\sigma_1)>0$,
  $\cot(\sigma_2)<0$, $\cot(\sigma_3)>0$, $\cot(\sigma_4)<0$ and then
  \begin{equation*}
    \sum_{i=1}^4 (-1)^{i} \cot(\sigma_i)\left( \frac1{|T_i|}+\frac1{|T_{i+1}|}\right) < 0.
  \end{equation*}
  In the case of $\Nvertex=4$ and two vertices in the same quadrant as in
  Figure~\ref{fig:regularQuadrantVertex.B}, the macro-elements are
  regular in most cases. For example, let us consider, for simplicity,
  $|T_i|=1$ for all $i=1,...,4$. Since we are taking
  $\sigma_i<\sigma_{i+1}$ for $i=1,...,3$ and the cotangent function
  is decreasing, if we choose $q_1,q_2\in Q_1$ then
  $-\cot(\sigma_1)+\cot(\sigma_2) := s_{1,2} <0$. Moreover:
  \begin{itemize}
  \item If $q_3\in Q_2$ and $q_4 \in Q_3$, then $\cot(\sigma_3)<0$ and
    $\cot(\sigma_4)>0$. Therefore $-\cot(\sigma_3)+\cot(\sigma_4) :=
    s_{3,4} >0$ and for each angle $\sigma_3\in Q_2$ there is a unique angle 
    $\sigma_4\in Q_4$ such that $s_{1,2}+s_{3,4}=0$. In all other
    cases the singularity
    condition~(\ref{eq:instabilityMacroCondition}) does not hold.
  \item If $q_3\in Q_2$ and $q_4 \in Q_4$, then $\cot(\sigma_3)<0$ and
    $\cot(\sigma_4)<0$. Therefore $s_{1,2}+\cot(\sigma_4) <0$ and
    there is a unique $\sigma_3\in Q_3$ such
    that~(\ref{eq:instabilityMacroCondition}) holds.
  \item If $q_3\in Q_3$ and $q_4 \in Q_4$
    (Figure~\ref{fig:regularQuadrantVertex.B}), then
    $\cot(\sigma_3)>0$ and
    $\cot(\sigma_4)<0$. Therefore~(\ref{eq:instabilityMacroCondition})
    does not hold.
  \end{itemize}
  \newcommand{\macroelemB}{
    \begin{tikzpicture}
      \def\Nvertex{4} \def\ix{3}\def\iy{1.25} \def\vertices{ \coordinate
        [label=below:$q_0$] (q0) at (0,0); \coordinate [label=+45:
        $q_1$] (q1) at (0.75*\ix,0.5*\iy); \coordinate [label=above:
        $q_2$] (q2) at (-\ix*0.5,\iy); \coordinate [label=135: $q_3$]
        (q3) at (-\ix,-\iy); \coordinate [label=-45: $q_4$] (q4) at
        (\ix*0.5,-\iy); } \def\xPoints { \coordinate (x1) at ($
        (q1)!.5!(q0) $); \coordinate (x2) at ($ (q2)!.5!(q0) $);
        \coordinate (x3) at ($ (q3)!.5!(q0) $); \coordinate (x4) at ($
        (q4)!.5!(q0) $); } \def\showPoints{ \foreach \i in {0,...,\Nvertex}
        \fill [black] (q\i) circle (3pt); \foreach \i in {1,...,\Nvertex} {
          \draw [fill, lightgray] (x\i) circle (2.25pt); \draw (x\i)
          circle (2.25pt); } } \def\showLines{
        \draw (q1)--(q2) -- (q3) -- (q4) -- cycle;
        \foreach \i in {1,...,\Nvertex} { \draw (q\i) -- (q0); } }
      \def\showTriangles{ \node at ($ (q4)!.5!(q1) $)
        [label=-45:${T_1}$] {}; \node at ($ (q1)!.5!(q2) $)
        [label=above:${T_2}$] {}; \node at ($ (q2)!.5!(q3) $)
        [label=135:${T_3}$] {}; \node at ($ (q3)!.5!(q4) $)
        [label=below:${T_4}$] {}; } \def\showAxis{ \draw[->,
        color=lightgray] (-1.1*\ix, 0)--(1.1*\ix,0); \draw[->,
        color=lightgray] (0, -1.5*\iy)--(0,1.5*\iy); } \showAxis
      \vertices \xPoints \showPoints \showLines \showTriangles
    \end{tikzpicture}
  } \newcommand{\macroelemC}{
    \begin{tikzpicture}
      \def\Nvertex{4} \def\ix{2.25}\def\iy{1.8} \def\vertices{ \coordinate
        [label=-90:$q_0$] (q0) at (0,0); \coordinate [label=+45:$q_1$]
        (q1) at (0.97*\ix,0.22*\iy); \coordinate [label=above:$q_2$]
        (q2) at (0.22*\ix,0.97*\iy);
        \coordinate [label=215: $q_3$] (q3) at (-0.94*\ix,-0.25*\iy);
        \coordinate [label=-45: $q_4$] (q4) at (\ix*0.71,-0.71*\iy); }
      \def\xPoints { \coordinate (x1) at ($ (q1)!.5!(q0) $);
        \coordinate (x2) at ($ (q2)!.5!(q0) $); \coordinate (x3) at ($
        (q3)!.5!(q0) $); \coordinate (x4) at ($ (q4)!.5!(q0) $); }
      \def\showPoints{ \foreach \i in {0,...,\Nvertex} \fill [black] (q\i)
        circle (3pt); \foreach \i in {1,...,\Nvertex} { \draw [fill,
          lightgray] (x\i) circle (2.25pt); \draw (x\i) circle
          (2.25pt); } } \def\showLines{
        \draw (q1)--(q2) -- (q3) -- (q4) -- cycle;
        \foreach \i in {1,...,\Nvertex} { \draw (q\i) -- (q0); } }
      \def\showTriangles{ \node at ($ (q4)!.5!(q1) $)
        [label=-45:${T_1}$] {}; \node at ($ (q1)!.5!(q2) $)
        [label=above:${T_2}$] {}; \node at ($ (q2)!.5!(q3) $)
        [label=135:${T_3}$] {}; \node at ($ (q3)!.5!(q4) $)
        [label=below:${T_4}$] {}; } \def\showAxis{ \draw[->,
        color=lightgray] (-1.2*\ix, 0)--(1.3*\ix,0); \draw[->,
        color=lightgray] (0, -1.15*\iy)--(0,1.15*\iy); } \showAxis
      \vertices \xPoints \showPoints \showLines \showTriangles
    \end{tikzpicture}
  }
  \begin{figure}
    \centering
    \begin{tabular}{cc}
      \subfloat[Each boundary vertex is in a different
      quadrant.]{ 
        \label{fig:regularQuadrantVertex.A}
        \macroelemB } 
      &
      \subfloat[Equal size elements with $q_1$, $q_2$ in first quadrant,
      $q_3$ in third and $q_4$ in fourth one.]{ 
        \label{fig:regularQuadrantVertex.B}
        \macroelemC }
    \end{tabular}
    \caption{Two $\FEthreeSpacesP{2}{1}{1}$ Stokes-regular macro-elements
      with $n=4$}
    \label{fig:regularQuadrantVertex}
  \end{figure}
\end{remark}
\begin{remark}
  Arguing as in Theorem~\ref{theorem:P1bP1P1-mesh-stability}, the
  local result of Theorem~\ref{theorem:P2P1P1:macroelem} can be
  extended to a global result regarding inf-sup
  condition~(\ref{eq:stokes-inf-sup}) in uniformly unstructured
  meshes.  Anyway, as commented in
  Remark~\ref{rk:macro211-vs-macro-b11}, this case is more complicated
  because it depends also on validating uniformly
  condition~(\ref{eq:instabilityMacroCondition}). Hence, in principle,
  the number of meshes where \FEthreeSpacesP211 is stable is slightly
  less than in the \FEthreeSpacesP{1,b}11 case.  For ensuring
    stability, a generalization of Algorithm~\ref{alg:x-unstruct} may
    be applied. In this case, the algorithm shall be more complex due
    to the additional condition~(\ref{eq:instabilityMacroCondition}).
    But in any case, in all our numerical experiments in random
    (unstructured) meshes, results have been satisfactory without any
    kind of mesh post-processing. See next section for details.
\end{remark}

\subsection{Numerical simulations}

\begin{test}[\FEthreeSpacesP211 and cavity test]
  \label{test:211-cavity-test}
  \begin{figure}
    \centering
    \begin{tabular}{ccc}
      \subfloat[$\ph$ in unstructured mesh.]{
        \label{fig:P2P1P1-p-uns}
        \pgfimage[width=0.31\linewidth]{\imgdir/21-1-p-uns}}
      & &
      \subfloat[$\ph$ in a structured mesh.]{
        \label{fig:P2P1P1-p-str}
        \pgfimage[width=0.31\linewidth]{\imgdir/21-1-p-str}}
    \end{tabular}
    \caption{\FEthreeSpacesP{2}{1}{1} pressure in unstructured and
      structured meshes.}
    \label{fig:P2P1P1-cavity-test}
  \end{figure}  
  Firstly, the cavity test which is described in
  Test~\ref{test:b11-cavity-test} in Section~\ref{sec:P1bP1P1} has
  been repeated for \FEthreeSpacesP211 elements, using the same
  unstructured and structured meshes
  (cf.~Figures~\ref{fig:P1bP1P1-unstruct-mesh}
  and~\ref{fig:P1bP1P1-struct-mesh}). The results obtained, which can
  be seen in Figure~\ref{fig:P2P1P1-cavity-test}, show correct results
  for the pressure in unstructured meshes
  (cf.~Figure~\ref{fig:P2P1P1-p-uns}), but spurious pressure modes appear in
  structured ones (cf.~Figure~\ref{fig:P2P1P1-p-str}), as expected from the
  theoretical results given in Section~\ref{sec:P2P1-P1-2d} (see
  Remark~\ref{rk:not-stability-of-P2P1P1}).
\end{test}

\begin{test}[Non-uniformly unstructured mesh]
  \label{tst:P211-non-unif-unstruct-meshes}
  In order to present a numerical evidence of the necessity of uniform
  unstructured meshes for stability, we have developed a
  numerical experiment where the behaviour of \FEthreeSpacesP211
  elements is tested in a non uniformly unstructured mesh family which
  ``converge to a structured mesh'' when $h\to 0$.
  
  Specifically we considered five meshes, ${\cal T}_{h_i}$, $i=1, ...,
  5$, where the finer mesh, ${\cal T}_{h_5}$, is defined as in
  Figure~\ref{fig:xStr-yUnstr-mesh}, with $63$ intervals on the top
  boundary (thus $h_5\simeq 1/63$). About the other coarser four
  meshes,
  they are defined similarly (with $3$, $7$, $15$ and $31$ intervals
  on the top boundary, respectively) but a random displacement is
  added to the $x$--coordinate of its vertices, breaking their
  structure. The amplitude of this displacement decreases ($0.4$,
  $0.2$, $0.1$, $0.05$, respectively), so that ${\cal T}_{h_4}$ is
  very similar to ${\cal T}_{h_5}$.  This way, we have a sequence of
  meshes, ${\cal T}_{h_1}$,..., ${\cal T}_{h_4}$, which are (non
  uniformly) \yUnstructured and whose ``limit'', ${\cal T}_{h_5}$, is a
  ``weakly \yStructured'' mesh (see
  Remark~\ref{rk:not-stability-of-P2P1P1}).

  Then we have developed a program which computes estimates of the
  constants, $\beta_h$ related to the discrete inf-sup
  condition~(\ref{eq:stokes-inf-sup}) for this mesh family and
  \FEthreeSpacesP211 FE. This program is based on
  ASCoT~\cite{ArnoldRognes2009}, a Python module built on top of the
  FEniCS framework~\cite{FEniCS:LoggMardalEtAl2012a} with the purpose
  of facilitating the estimation of Brezzi inf-sup constants. Results
  (see Table~\ref{tab:non-uniform-unstr-meshes}) suggest that
  $\beta_h\to 0$ depends on $h$, vanishing when mesh structure
  appears, therefore discrete inf-sup condition is not satisfied.
  
  \begin{table}
    \centering
    \small
    \begin{tabular}{l@{\qquad}r@{\qquad}cccccc}
      & 
      & ${\cal T}_{h_1}$ & ${\cal T}_{h_2}$ & ${\cal T}_{h_3}$ & ${\cal T}_{h_4}$ & ${\cal T}_{h_5}$ 
      \\ \otoprule
      $h$ & 
      & $1/3$ & $1/7$ & $1/15$ & $1/31$ & $1/63$ 
      \\ \midrule
      $\beta_h$ &
      & $0.19384$ & $0.096089$ & $0.048363$ & $0.023503$ & $0.0086823$
      \\ \bottomrule
    \end{tabular}
    \caption{Vanishing discrete inf-sup constants on a non-uniformly unstructured
      mesh family}
    \label{tab:non-uniform-unstr-meshes}
  \end{table}
\end{test}

\begin{test}[\FEthreeSpacesP211 and error orders]
  \label{test:21-1-error-orders}

  \begin{figure}
    \centering
    \begin{tabular}{@{}c@{}c@{}}
      \subfloat
       {
        \pgfimage[width=0.49\textwidth]{\imgdir/21-1-errors-loglog}
        \label{fig:21-1-errors-loglog}
      }
      &
      \subfloat
      {
        \pgfimage[width=0.49\textwidth]{\imgdir/22-1-errors-loglog}
        \label{fig:22-1-errors-loglog}
      }
    \end{tabular}
    \caption{Velocity and pressure errors for \FEthreeSpacesP211 (left) and
      \FEthreeSpacesP221 (right)}
  \end{figure}

  The experiment described in
  Test~\ref{test:b1-1-error-orders} has been reproduced for \FEthreeSpacesP211.
  Figure~\ref{fig:21-1-errors-loglog} shows the absolute
  errors and orders obtained. The following conclusions can
  be summarized:
  \begin{enumerate}
  \item For both components of the velocity, order $O(h^2)$ in
    $L^2(\Omega)$ and order $O(h)$ in $H^1(\Omega)$ are obtained,
    while order $O(h)$ for pressure is suggested.
  \item Optimal order $O(h^3)$ in $L^2(\Omega)$ and $O(h^2)$ in
    $H^1(\Omega)$ is not reached for $u$ (although it is approximated
    in \P2), due to the effect of approximating $v$ in \P1.
    Therefore, a loss of accurate is produced passing from the
    classical Taylor-Hood element \FEthreeSpacesP221
    (Figure~\ref{fig:22-1-errors-loglog}) to \FEthreeSpacesP211
    (Figure~\ref{fig:21-1-errors-loglog}). In spite of it, this latter
    element \FEthreeSpacesP211 is more stable than \FEthreeSpacesP221
    in degenerated anisotropic problems, as is shown
    in~\CiteSecondPaper.
  \end{enumerate}
  \begin{table}
    \centering
    \small
    \begin{tabular}{l@{\qquad}r@{\qquad}cccccc}
      \\ \toprule
      & $h$
      & $2^{-3}$ & $2^{-4}$ & $2^{-5}$ & $2^{-6}$ & $2^{-7}$  & $2^{-8}$
      \\ \otoprule
      \multirow{2}*{$\uu$} 
      &  $\|u-u_h\|_{L^2}$ &  
      2.185  & 2.291  & 2.008  & 2.029  & 2.031  & 2.066
      \\ \cmidrule{2-8}
      & $\|u-u_h\|_{H^1_0}$ & 
      1.758  & 1.774  & 0.986  & 1.212  & 1.067  & 1.237
      \\ \midrule
      \multirow{2}*{$\vv$} 
      & $\|v-v_h\|_{L^2}$ & 
      2.056  & 2.116  & 2.028  & 1.980  & 2.068  & 2.040
      \\ \cmidrule{2-8}
      & $\|v-v_h\|_{H^1_0}$ & 
      0.977  & 1.044  & 1.017  & 0.990  & 1.022  & 1.019
      \\ \midrule
      {$\pp$}
      & $\|p-p_h\|_{L^2}$ &  
      1.070  & 2.101  & 0.798  & 1.419  & 0.599  & 0.820
      \\ \bottomrule
    \end{tabular}
    \caption{Error orders for velocities and pressure
      (\FEthreeSpacesP211 FE)}
    \label{tab:21-1-orders}
  \end{table}

\end{test}

\begin{test}[$3D$ domain, \FEfourSpacesP2211 or \FEfourSpacesP2111 and cavity tests]
  Finally, we have repeated the $3D$ experiment which was described in
  Test~\ref{test:b11-cavity-test}, using both structured and
  unstructured meshes, but replacing the bubble by \P2.

  Owing to the choice of a color map that amplifies the
  variations around small values of pressure, spurious oscillations of
  the $\ph$ are evident in the structured mesh for \FEfourSpacesP2211
  and \FEfourSpacesP2111 
  (Figures~\ref{fig:P2P1P1-strucured-cavity-test-3d-b}
  and~\ref{fig:P2P1P1-strucured-cavity-test-3d-c}).
  \begin{figure}
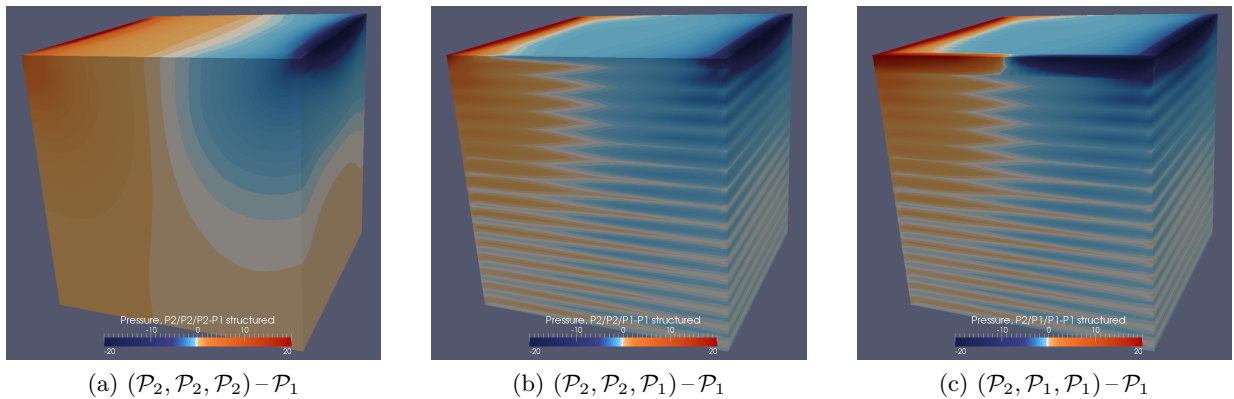

    \centering
    \begin{tabular}{ccc}
      \subfloat[\FEfourSpacesP2221]{ 
        \pgfimage[width=0.3\linewidth]{\imgdir/222-1-str}
        \label{fig:P2P1P1-strucured-cavity-test-3d-a}}
      &
      \subfloat[\FEfourSpacesP22{1}1]{ 
        \pgfimage[width=0.3\linewidth]{\imgdir/221-1-str}
        \label{fig:P2P1P1-strucured-cavity-test-3d-b}}
      &
      \subfloat[\FEfourSpacesP2{1}{1}1]{ 
        \pgfimage[width=0.3\linewidth]{\imgdir/211-1-str}
        \label{fig:P2P1P1-strucured-cavity-test-3d-c}}
    \end{tabular}
    \caption{Comparing the pressure  in
      a $3D$ \emph{structured mesh}.}
    \label{fig:P2P1P1-strucured-cavity-test-3d}
  \end{figure}  

  \begin{figure}
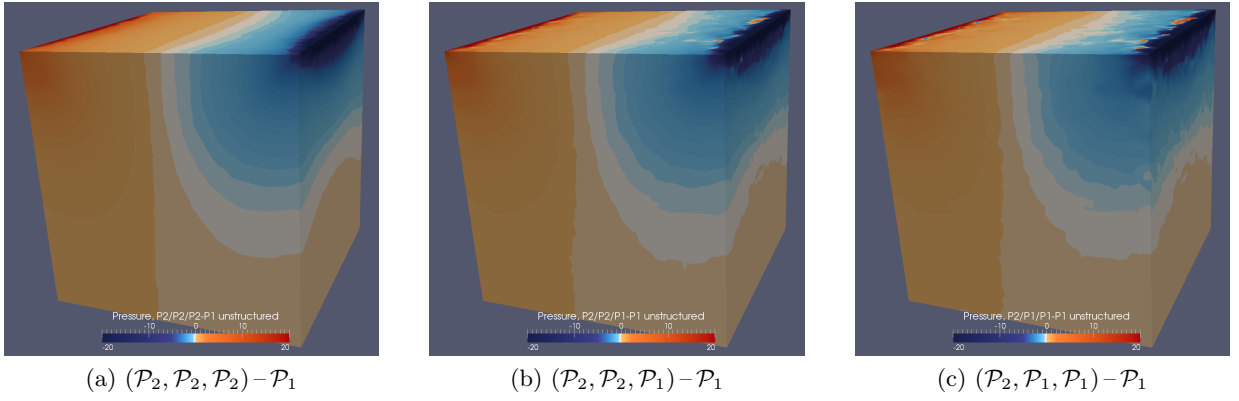

    \centering
    \begin{tabular}{ccc}
      \subfloat[\FEfourSpacesP2221]{ 
        \pgfimage[width=0.3\linewidth]{\imgdir/222-1-uns}
        \label{fig:P2P1P1-unstructured-cavity-test-3d-a}}
      &
      \subfloat[\FEfourSpacesP22{1}1]{ 
        \pgfimage[width=0.3\linewidth]{\imgdir/221-1-uns}
        \label{fig:P2P1P1-unstructured-cavity-test-3d-b}}
      &
      \subfloat[\FEfourSpacesP2{1}{1}1]{ 
        \pgfimage[width=0.3\linewidth]{\imgdir/211-1-uns}
        \label{fig:P2P1P1-unstructured-cavity-test-3d-c}}
    \end{tabular}
    \caption{Comparing the pressure in
      a  $3D$ \emph{unstructured mesh}.}
    \label{fig:P2P1P1-unstructured-cavity-test-3d}
  \end{figure}

\end{test}

\section{Instability of \FEthreeSpaces{\Q2}{\Q1}{\Q1} in structured meshes}
\label{sec:Q2Q1Q1}

In this section we present a counterexample showing that the
combination \FEthreeSpaces{\Q2}{\Q1}{\Q1} is not stable in structured quadrangular 
macro-elements. As above, this result can be extended to global quadrangular meshes.

Let us consider a $2D$ quadrangular  structured mesh $\Th$. Let $M$ be a  macro-element as in
Figure~\ref{fig:Q2Q1Q1-macroelement}, i.e.~$M$ is union of four
rectangular elements defined by $9$ vertices, which are
denoted as $q_{ij}$ ($i,j=1, ..., 3$) and constitute the \Q1 degrees
of freedom. The rest of the \Q2 degrees of freedom are located at the
midpoints of the edges and at the center of the quadrangles.  Without
loss of generality it can be assumed that the central vertex is
located at the origin of coordinates; i.e.~$q_{22}=(0,0)$ using the notation of
Figure~\ref{fig:Q2Q1Q1-macroelement}.

Let us define the following FE spaces in $M$:
\begin{align*}
  \UM &= \{\uh \in H_0^1(M)\cap C^0(\overline M) \st \uh|_K \in \Q2,
  \ \forall K\in M
  \}, 
  \\ 
  \VM & = \{ \vh\in H_0^1(M) \cap C^0(\overline M) \st
  \vv|_K\in \Q1, \ \forall K\in M \}
  \\ 
  \PM &= \{ \ph\in  C^0(\overline M) \st \pp|_K\in \Q1,
  \ \forall K\in M \}.
\end{align*}

\begin{figure}
  \centering
  \begin{tikzpicture}[>=latex]
    \def\ix{40}\def\iy{20}
    \def\vertices{
      \foreach \i in {1,...,3} {
        \foreach \j in {1,...,3} {
          \coordinate (q\i\j) at (\ix*\i,\iy*\j);
        }
      }
    }
    \def\showLines{
      \draw (q11) -- (q21) -- (q22) -- (q12) -- cycle;
      \draw (q21) -- (q31) -- (q32) -- (q22) -- cycle;
      \draw (q12) -- (q22) -- (q23) -- (q13) -- cycle;
      \draw (q22) -- (q32) -- (q33) -- (q23) -- cycle;
    }
    \def\showNodes{
      \draw[black, fill] ($ (q11)!.5!(q22) $) circle(1pt);
      \draw[black, fill] ($ (q22)!.5!(q33) $) circle(1pt);
      \draw[black, fill] ($ (q13)!.5!(q22) $) circle(1pt);
      \draw[black, fill] ($ (q31)!.5!(q22) $) circle(1pt);
      \foreach \i in {1,...,3} {
        \draw[black] ($ (q\i1)!.5!(q\i2) $) circle(1.5pt);
        \draw[black] ($ (q\i2)!.5!(q\i3) $) circle(1.5pt);
        \draw[black] ($ (q1\i)!.5!(q2\i) $) circle(1.5pt);
        \draw[black] ($ (q2\i)!.5!(q3\i) $) circle(1.5pt);
        \draw[black] ($ (q2\i)!.5!(q3\i) $) circle(1.5pt);
        \foreach \j in {1,...,3} {
          \draw [black, fill] (q\i\j) circle (2.5pt);
          \node (texto) at (q\i\j) [label=225:{\small $q_{\i\j}$}] {};
        }
      }      
    }
    \showLines
    \showNodes
  \end{tikzpicture}  
  \caption{\FEthreeSpaces{\Q2}{\Q1}{\Q1} macroelement}
  \label{fig:Q2Q1Q1-macroelement}
\end{figure}

Let us consider the $\Q1$ function $\ph$ defined as 
  \begin{equation}
    \ph(x,y)=
    \begin{cases}
      a+ cy & \text{if } y>0, \\
      a-cy & \text{if } y\le 0,
    \end{cases}
    \label{eq:pressure-counterexample}    
  \end{equation}
  where $a,c\in \Rset$, $c\neq 0$. 
Then each $\vh\in\VM$ can be written as $\alpha\, \phi_{22}$ for any
$\alpha\in\Rset$, where $\phi_{22}$ is the $\Q1$ basis function that
is equal to $1$ in $q_{22}$ and zero in all other vertices of $M$.
Hence, if we apply the 2D trapezoidal rule (which is exact in $\Q1$) and taking into account that $\partial _x p_h =0$,
\begin{align*}
  \int_M (\dx\uh+\dy\vh) \ph 
  &= - \int_{M} \vv_h \; \dy \ph 
  =  -\frac{c}{|M^+|} \int_{M^+} \alpha\phi_{22} + \frac{c}{|M^-|} \int_{M^-} \alpha\phi_{22} 
  \\  &=
  \frac{c}{|M^+|} \sum_{K\subset {M^+}} \frac{|K|}{4} \alpha
  -  \frac{c}{|M^-|} \sum_{K\subset {M^-}} \frac{|K|}{4} \alpha
  = \frac{c}{4}\alpha - \frac{c}{4} \alpha  = 0.
\end{align*}
Therefore the combination $\FEthreeSpaces{\Q2}{\Q1}{\Q1}$
is not stable in quadrangular macro-elements.

As in Section~\ref{sec:P2P1P1}, it is not difficult to extend the
previous counterexample to demonstrate that there exists a family of not
null pressures $\ph\in\Ph$ satisfying $ (\div(u_h,v_h),p_h)=0$ for all
$\uh\in\Uh$ and  $\vh\in\Vh$. Hence, in rectangular structured
meshes, the Stokes inf-sup condition~(\ref{eq:stokes-inf-sup}) does
not hold for the $\FEthreeSpaces{\Q2}{\Q1}{\Q1}$ combination.

\begin{test}[\FEthreeSpaces{\Q2}{\Q1}{\Q1} and cavity test]
  Numerical experiments confirm the $\FEthreeSpaces{\Q2}{\Q1}{\Q1}$
  instability, as can be seen in
  Figure~\ref{fig:Q2Q1Q1-instability}. We consider the $2D$ Stokes
  problem~(\ref{eq:Stokes.a})--(\ref{eq:Stokes.b}) in
  $\domain=(0,1)^2$.  Dirichlet boundary conditions for $\uu$ have
  been applied; $\uu=1$ on the top $\{y=1\}$ and $\uu=0$ on the rest
  of $\partial\domain$. About $\vv$, a natural boundary condition
  $\frac{\partial\vv}{\partial \mathbf{n}}=0$ has been chosen on
  $\partial\domain$. Both stable \FEthreeSpaces{\Q2}{\Q2}{\Q1} and
  unstable \FEthreeSpaces{\Q2}{\Q1}{\Q1} combinations of FE have been
  tested in a mesh composed of $15\times 15$ rectangular elements. As
  expected, Figure~\ref{fig:Q2Q1Q1-instability} shows spurious
  pressure oscillations in the latter case.
  This experiment has been programmed in C++
  with~\textit{LibMesh}~\cite{LibMesh:06}, a parallel FE
  library which supports natively the $\Q1$ and $\Q2$ elements.

  \begin{figure}
    \centering
    \begin{tabular}{cc}
      \subfloat[ ]
      {\pgfimage[width=0.4\linewidth]{\imgdir/pressure-Q2Q2Q1}}
      &
      \subfloat[ ]
      {\pgfimage[width=0.4\linewidth]{\imgdir/pressure-Q2Q1Q1}}
    \end{tabular}
    \caption{Pressure for \FEthreeSpaces{\Q2}{\Q2}{\Q1} (left) and
      \FEthreeSpaces{\Q2}{\Q1}{\Q1} (right). The latter, presents
      spurious ocillations.}
    \label{fig:Q2Q1Q1-instability}
  \end{figure}

\end{test}

\section{Conclusions}
In this paper we have developed two new families of FE spaces
approximating the Stokes problem, where continuous and piecewise
linear functions $\P1$ are enriched, only for one component of the
velocity, by either bubble $\P{1,b}$ or quadratic polynomial
$\P2$. These combinations have been denoted, in the $2D$ case, by
\FEthreeSpacesP{1,b}11 and \FEthreeSpacesP{2}11.

In order to compare the use of bubble functions, we have seen that
combination \FEthreeSpacesP{1,b}11 in uniformly unstructured meshes is more
efficient than the well-known mini-element \FEthreeSpacesP{1,b}{1,b}1,
because \FEthreeSpacesP{1,b}11 preserves the first order accurate of
\FEthreeSpacesP{1,b}{1,b}1 while requires a minor number of degrees of
freedom and then a smaller computational effort.

On the other hand, we have also proved that combination
\FEthreeSpacesP{2}11 in most uniformly unstructured meshes is stable but it
does not preserve the second order accuracy of the well-known
Taylor-Hood element \FEthreeSpacesP{2}{2}1, although
\FEthreeSpacesP{2}11 uses a minor number of degrees of freedom.

Moreover, the stability constraints for \FEthreeSpacesP{2}11 are more
restrictive than for \FEthreeSpacesP{1,b}11, because they depend not
only on the structure of the mesh but also depend on the algebraic
condition~(\ref{eq:sumAlphaiCi}). Hence, in principle, the number
of meshes where \FEthreeSpacesP211 is stable is slightly less than in the
\FEthreeSpacesP{1,b}11 case.
  
Finally, the same type of results are also deduced in $3D$ domains for
the case of bubble functions.

\footnotesize 
\bibliography{biblio}

\newcommand{\etalchar}[1]{$^{#1}$}
\begin{thebibliography}{LMW{\etalchar{+}}12}

\bibitem[AG01]{Azerad-Guillen:01}
P.~Az{\'e}rad and F.~Guill{\'e}n.
\newblock Mathematical justification of the hydrostatic approximation in the
  primitive equations of geophysical fluid dynamics.
\newblock {\em Siam J. Math. Ana.}, 33(4):847--859, 2001.

\bibitem[AR09]{ArnoldRognes2009}
Douglas~N. Arnold and Marie~E. Rognes.
\newblock Stability of {L}agrange elements for the mixed {L}aplacian.
\newblock {\em Calcolo}, 2009.

\bibitem[Az{\'e}94]{Azerad:1994}
P.~Az{\'e}rad.
\newblock Analyse et approximation du probl\`eme de {S}tokes dans un bassin peu
  profond.
\newblock {\em C. R. Acad. Sci. Paris S\'er. I Math.}, 318(1):53--58, 1994.

\bibitem[Az{\'e}96]{Azerad:PhD:96}
P.~Az{\'e}rad.
\newblock {\em Analyse des équations de {N}avier-{S}tokes en bassin peu
  profond et de l'équation de transport}.
\newblock PhD thesis, Neuchâtel, 1996.

\bibitem[BBF13]{boffi-brezzi-fortin_mixed_2013}
Daniele Boffi, Franco Brezzi, and Michel Fortin.
\newblock {\em Mixed Finite Element Methods and Applications}, volume~44 of
  {\em Springer Series in Computational Mathematics}.
\newblock Springer Berlin Heidelberg, Berlin, Heidelberg, 2013.

\bibitem[BF91]{Brezzi-Fortin:91}
F.~Brezzi and M.~Fortin.
\newblock {\em Mixed and Hybrid Finite Element Methods}.
\newblock Springer-Verlag, New-York, 1991.

\bibitem[BL92]{Besson-Laydi:92}
O.~Besson and M.R. Laydi.
\newblock Some estimates for the anisotropic {N}avier-{S}tokes equations and
  for the hydrostatic approximation.
\newblock {\em Math. Mod. and Num. Anal}, Vol. 26(7):855--865, 1992.

\bibitem[CB09]{CushmanRoisin-Beckers:09}
B.~{Cushman-Roisin} and J.~M. Beckers.
\newblock {\em Introduction to Geophysical Fluid Dynamics - Physical and
  Numerical Aspects}.
\newblock Academic Press, 2009.

\bibitem[GR86]{girault-raviart:86}
V.~Girault and P.-A. Raviart.
\newblock {\em Finite element methods for {Navier-Stokes} equations}.
\newblock Springer-Verlag, 1986.

\bibitem[GR09]{Gmsh:09}
C.~Geuzaine and J.-F. Remacle.
\newblock Gmsh: {A 3-D} finite element mesh generator with built-in pre- and
  post-processing facilities.
\newblock {\em International Journal for Numerical Methods in Engineering},
  79:1309 -- 1331, 2009.

\bibitem[GR14]{Guillen-RGalvan:13b}
F.~{Guill\'en-Gonz\'alez} and J.R. {Rodr\'{\i}guez\mbox-Galv\'an}.
\newblock Analysis of the hydrostatic stokes problem and finite-element
  approximation in unstructured meshes.
\newblock {\em Numer. Math.}, 2014.
\newblock Accepted, doi 10.1007/s00211-014-0663-8.

\bibitem[KPSC06]{LibMesh:06}
B.S. Kirk, J.W. Peterson, R.H. Stogner, and G.F. Carey.
\newblock libmesh: a c++ library for parallel adaptive mesh
  refinement/coarsening simulations.
\newblock {\em Eng. with Comput.}, 22(3):237--254, December 2006.

\bibitem[LMW{\etalchar{+}}12]{FEniCS:LoggMardalEtAl2012a}
Anders Logg, Kent-Andre Mardal, Garth~N. Wells, et~al.
\newblock {\em Automated Solution of Differential Equations by the Finite
  Element Method}.
\newblock Springer, 2012.

\bibitem[PHLHM]{FreeFem++}
O.~Pironneau, F.~Hecht, A.~Le~Hyaric, and J.~Morice.
\newblock {FreeFEM++, \url{http://www.freefem.org/}}.

\bibitem[QZ07]{QinZhang:07}
J.~Qin and S.~Zhang.
\newblock Stability and approximability of the {P1-P0} element for {Stokes}
  equations.
\newblock {\em International Journal for Numerical Methods in Fluids},
  54:497--515, 2007.

\bibitem[Ste84]{Stenberg:84}
R.~Stenberg.
\newblock Analysis of mixed finite element methods for the {Stokes} problem:
  {A} unified approach.
\newblock {\em Math Comput}, 42(165):9--23, January 1984.

\bibitem[Ste90]{Stenberg:90}
R.~Stenberg.
\newblock A technique for analysing finite elements methods for viscuous
  incompressible flow.
\newblock {\em International Journal for Numerical Methods in Fluids},
  11:835--948, 1990.

\end{thebibliography}


\end{document}